\documentclass{article}
\usepackage[margin = 3cm, bottom = 3cm]{geometry}
\usepackage[all]{xy}
\usepackage{makeidx}
\usepackage{changes}
\usepackage[mathscr]{eucal}

\title{The holonomy of a singular leaf}

\author{ Camille Laurent-Gengoux\thanks{Institut Elie Cartan de Lorraine, UMR 7502, Universit\'e de Lorraine,  France},
Leonid Ryvkin 
\thanks{Georg-August-Universit\"at G\"ottingen,
Institut f\"ur Mathematik, Germany}
\thanks{Institut Math\'ematiques de Jussieu, Universit\'e Paris Diderot, Paris, France.} \thanks{Faculty of mathematics, Universit\"at Duisburg-Essen, Essen, Germany }}

\AtEndDocument{\bigskip{\footnotesize
  \textsc{C. Laurent-Gengoux},  \textrm{ Institut Elie Cartan de Lorraine (UMR 7502), Universit\'e de Lorraine, 3 rue Augustin Fresnel, 57000 Metz, Technop\^ole, France}. \par \texttt{camille.laurent-gengoux@univ-lorraine.fr} \par
  \addvspace{\medskipamount}
  \textsc{L. Ryvkin},  \textrm{
Georg-August-Universit\"at G\"ottingen,
Institut f\"ur Mathematik
Bunsenstr. 3-5
37073 G\"ottingen, Germany}. \par \texttt{leonid.ryvkin@mathematik.uni-goettingen.de }
}}

\usepackage{latexsym,amsfonts,amsthm,amsmath,amscd,amssymb,color}
\usepackage{cite}

\newtheorem{lemma}{Lemma}[section]
\newtheorem{flemma}[lemma]{Fundamental Lemma}
\newtheorem{theorem}[lemma]{Theorem}
\newtheorem{proposition}[lemma]{Proposition}
\newtheorem{corollary}[lemma]{Corollary}

\newtheorem{definition}[lemma]{Definition}
\newtheorem{convention}[lemma]{Convention}
\newtheorem*{convention*}{Convention}
\newtheorem{example}[lemma]{Example}
\newtheorem{remark}[lemma]{Remark}

\newtheorem*{question}{Question}

\date{}

\newcommand{\Cinfty}[1]{{\mathcal C^\infty_{#1} }}
\newcommand{\fol}[1]{{\mathfrak #1}}
\newcommand{\sh}[1]{{\mathcal O_\bullet^{#1}}}
\newcommand{\NQ}[1]{{\mathcal #1}}
\makeindex
\begin{document}

\maketitle

\begin{abstract}
	We introduce the holonomy of a singular leaf $L$ of a singular foliation as a sequence of group morphisms from $\pi_n(L)$ to the $\pi_{n-1}$ of the universal Lie $\infty$-algebroid of the transverse foliation of $L$. We
	include these morphisms in a long exact sequence, 
	thus relating them to the holonomy groupoid of Androulidakis and Skandalis and to a similar construction by Brahic and Zhu for Lie algebroids.
\end{abstract}

\tableofcontents

\setcounter{section}{0}
\newpage

\section*{Introduction}
The holonomy of a leaf $L$ of a regular foliation $ {\fol{F}}$ on a manifold $M$ is defined \cite{MR47386, MR0055692,MR189060} as a group morphism:
\begin{equation} \label{eq:holonomy} Hol : \pi_1(L,\ell) \mapsto {\mathrm{Diff}}_{\ell}(T_\ell), \end{equation}
where $\pi_1(L),\ell$ is the fundamental group of $L$ at $\ell \in L$ and ${\mathrm{Diff}}_{\ell}(T_\ell)$ is the group of germs of diffeomorphisms of a transversal $T_\ell$ of $L$ at a point $\ell \in L$. Assuming the existence of Ehresmann connections on a tubular neighborhood of $L$, the group morphism \eqref{eq:holonomy} then takes values in the group of diffeomorphisms of some fixed transversal. The construction of $Hol$ consists in lifting loops on $ L$ based at $\ell $ to a diffeomorphism  $T_\ell \simeq  T_{\ell}$, then in showing that homotopic loops induce the same diffeomorphism. \\

Consider now a singular foliation \cite{MR0149402,Cerveau,D,AS09}, \emph{i.e.}~a ``locally finitely generated  sub-module $ {\fol{F}}$ of the module of vector fields, stable under Lie bracket''.
By Hermann's theorem \cite{MR0142131}, singular foliations induce a partition of $M$ into submanifolds called leaves of the foliation. Unlike for regular foliations, the leaves may not be all of the same dimension. 
Those leaves around which the dimension jumps are called \emph{singular leaves}. This leads us to the natural:

\begin{question}
	\label{ques:main}
	How to define an analogue of the holonomy \eqref{eq:holonomy} for a singular leaf of a singular foliation?
\end{question}

Dazord \cite{Dazord} proposed in 1985 an answer, which consists of a group morphism from the fundamental group of the leaf to the bijections of the orbit space of the transversal. Recently, Androulidakis and Zambon \cite{AZ13,AZ2} used the holonomy Lie groupoid of a singular foliation \cite{AS09} as a replacement for the fundamental group for a singular leaf.  Using the universal Lie $ \infty$-algebroid  of a singular foliation discovered in \cite{LLS}, we propose a construction of ``higher holonomies'' taking into account the higher homotopy groups of the singular leaves. The first of these holonomies is the one defined by Dazord in \cite{Dazord}. We are also able relate our construction to Andoulidakis and Zambon's holonomy, but our construction differs from theirs. For leaves reduced to a point, our construction becomes trivial, while their construction still makes sense.

 To explain the construction, recall that to most \footnote{More precisely, the universal Lie $\infty$-algebroid is shown in \cite{LLS} to exist for any singular foliation that admits a geometric resolution. This happens in particular for locally real analytic singular foliations, that is singular foliations that have, in a neighborhood of every point, generators which are real analytic in some local coordinates. This class is quite large. The construction has been later extended by the first author and Ruben Louis (see \cite{LGLo}) to arbitrary Lie-Rinehart algebras (see also \cite{1811.03078}), in particular to general singular foliations, at the cost that the Lie $\infty$-algebroid could have infinite dimension.} singular foliations ${\fol{F}}$ on $M$, Sylvain Lavau, Thomas Strobl and C.L.  associated in \cite{LLS} a Lie $\infty$-algebroid ${\NQ{U}}^{\fol{F}}$ {concentrated in negative degree}, constructed over  {projective} resolutions of $ {\fol{F}}$ as a module over smooth functions on $M$. Although the construction of ${\NQ{U}}^{\fol{F}}$ relies on several choices, it is unique up to homotopy.
In particular, its homotopy groups $\pi_n( {\NQ{U}}^{\fol{F}},m) $,
at a point $m \in M$, \emph{i.e.}~homotopy classes of maps from $S^n$ to $ {\NQ{U}}^{\fol{F}}$ mapping the north pole to $m$, depend only on the singular foliation $\fol{F}$. We can therefore denote these groups as $ \pi_n({\fol{F}},m)$, without any reference to the universal Lie $\infty$-algebroid ${\NQ{U}}^{\fol{F}}$. Altogether, these groups form a bundle of groups over $M$: its group of (smooth) sections shall be denoted by $\Gamma(\pi_n({\fol{F}} ))$. We denote its group of smooth sections by $\Gamma(\pi_n(\fol{F}))$.
For $n=0$, by convention, $\Gamma(\pi_0({\fol{F}} ))$ consists in bijections of the orbit space  $ M/\fol{F}$ that come from a leaf-preserving smooth diffeomorphism of $M$.
Now, let us choose a leaf $L$. Every submanifold $T$ transverse to $L$ comes equipped with a singular foliation, which does not depend on the choice of $T$, see \cite{Cerveau,Dazord,AS09}.
It is called the \emph{transverse foliation} and we denote it by $ {\fol{T}}_\ell$ , with $ \ell $ the intersection point $T \cap L$. 
We claim that a reasonable answer to the question above is to define the higher holonomies as a sequence of group morphisms:
 \begin{equation}\label{eq:intro:Hol}
     Hol : \pi_n (L,\ell) \mapsto \Gamma(\pi_{n-1}  ({\fol{T}}_\ell )),
 \end{equation} 
 that we show in the present article to exist, provided that the leaf $L$ admits a complete Ehresmann connection. If no Ehresmann connection is complete, then the constructions would still make sense, but at the level of germs.
We also justify the interest of the definition by showing that these morphisms belong to a long exact sequence:
  \begin{equation}\label{eq:intro:long}
 \dots \to  \pi_{n+1}(L,\ell) \stackrel{Hol^{}}{  \longrightarrow }  \pi_n({\fol{T}}_\ell , m) \to  \pi_n({\fol{F}} , {m}) \to \pi_n(L,\ell) \stackrel{Hol^{}}{  \longrightarrow }  \pi_{n-1}({\fol{T}}_\ell , m) \to \dots  .
 \end{equation}
 for all $m \in p^{-1}(\ell)$.
 
Indeed, the existence of such a long exact sequence is a general phenomenon for NQ-manifold fibrations and is proven in this context, see Theorem \ref{thm:snake}. This extends a theorem of Olivier Brahic and Chenchang Zhu \cite{Zhu} about Lie algebroids, which was an important inspiration for the present work. This construction is quite involved and is addressed in Section \ref{sec:nq}.

In Section \ref{sec:leaf}, we apply the results of Section \ref{sec:nq} to the universal Lie $\infty$-algebroid of a singular foliation: we first carefully describe the contexts where this construction is possible, and introduce the notion of (complete) $ \fol{F}$-connection. Eventually, we construct the higher holonomies \eqref{eq:intro:Hol} that {fit into} the exact sequence \eqref{eq:intro:long}.

We finish this article by relating our higher holonomies $ Hol$ with two types of Lie algebroids and groupoids. On the one hand, we show that our higher holonomies are universal in the following sense: Assume that there exists a Lie groupoid $ \mathbf G$ whose leaves are the leaves of $ \fol{F}$.
Then, the above-mentioned theorem by {Brahic-Zhu} implies that for every leaf $L$, {there exist} group morphisms:
 $$ \delta_{BZ} \colon \pi_n(L,\ell)  \mapsto  \Gamma(\pi_{n-1} ( \mathbf K)) $$
valued in sections of a bundle of group $ \pi_{n-1}( \mathbf K)$ whose precise sense shall be given in due time. There are canonical group morphisms $ \psi : \Gamma(\pi_{n-1} ( \mathbf K))  \to  \Gamma(\pi_{n-1}(\fol{T}_\ell)) $ by the universality of $\NQ{U}^\fol{F} $. We show that the higher holonomies $ Hol$ coincide with the composition $Hol = \psi \circ \delta_{BZ} $. This must be seen as a universality condition.

On the other hand, we claim that our higher holonomies are a refinement of holonomies naturally associated to Androulidakis and Skandalis holonomy groupoid $\mathbf T $ of the transverse singular foliation $ \fol{T}_\ell$. More precisely, we show that there exist natural group morphisms
 $$ Hol_{AS} : \pi_n(L,\ell) \to \Gamma(\pi_{n-1} ( \mathbf T))$$
  and canonical group morphisms $ \phi :  \Gamma(\pi_{n-1}(\fol{T}_\ell))  \to \Gamma(\pi_{n-1} ( \mathbf T)) $ (that correspond to  $1$-truncation of the universal NQ-manifold $\NQ{U}^\fol{F} $) such that $ Hol_{AS} = \phi \circ Hol $.

\begin{convention*}
For $\mathbf G$ a groupoid over a manifold $M$, and $ X,Y \subset M$, we use the following notations:
$$ \mathbf G|_X =s^{-1}(X) \hbox{ , }  \mathbf G|^Y= t^{-1}(Y) \hbox{ and }  \mathbf G|_X^Y =s^{-1}(X) \cap t^{-1}(Y). $$
\end{convention*}

\noindent {\textbf{Acknowledgments.}} We acknowledge crucial discussions with Iakovos Androulidakis (for  $ \fol{F}$-connections) and Pavol \v{S}evera (for fundamental groups of NQ-manifolds): We will mark them in the text in due places. We also thank the workshop and conference ``Singular Foliations'' in Paris Diderot and Leuven where the content of the article was presented. We also express special gratitude to Sylvain Lavau, Thomas Strobl, Marco Zambon and Chenchang Zhu. L. R. was supported by the Ruhr University Research School PLUS, funded by Germany’s Excellence Initiative [DFG GSC 98/3] and the PRIME programme of the
German Academic Exchange Service (DAAD) with funds from the German Federal Ministry of Education and Research (BMBF). Both authors were supported by CNRS MITI 80 Prime projet ``Granum''.
 {Finally, we would like to thank the referee for his  suggestions to clarify the article and improve the exposition.}

\section{Homotopy groups of NQ-manifolds}  
\label{sec:nq}

\subsection{On NQ-manifolds and Lie $\infty$-algebroids}

\subsubsection{NQ-manifolds, morphisms and  homotopies}

 {
An \emph{N-manifold} structure over a manifold $M$ is a sheaf of $\mathbb N^0$-graded commutative algebras $\sh{}$, called sheaf of functions, such that for all $p\in M$, there exists an open neighborhood $U$ and a negatively graded vector bundle $E_\bullet=\bigoplus_{i=1}^dE_{-i}$ over $U$ of finite rank, such that $\sh{}|_U$ is isomorphic to $ \Gamma(\cdot, E_\bullet)$. \\

An NQ-manifold $\NQ{M}=(M, \sh{\NQ{M}},Q_\NQ{M})$ is an N-manifold equipped with a degree $+1$ derivation of its sheaf of functions, squaring to zero.}

\begin{example}
\normalfont
Any manifold $ \Sigma$ induces an NQ-manifold denoted by $T[1]\Sigma$: Its functions are the  exterior forms $\Omega^\bullet_\Sigma $ equipped with the de Rham differential  ${\mathrm d}_{dR}$.
\end{example}

\begin{definition}
A  \emph{negatively graded Lie $\infty$-algebroid} over a manifold $M$ is a graded vector bundle $(E_{-i})_{i=1}^d  $ concentrated in degrees ranging from $-1$ to $-d$ whose sections come equipped  with a graded symmetric Lie $\infty $-algebra structure $\{l_k\}_{k \geq 1} $, such that the $k$-ary brackets $l_k$ are $\Cinfty{M}$-linear, except for $l_2 $ which satisfies:
 $$ l_2(x,fy) \,  = \,  f \,  l_2 (x,y)  \, +  \, \rho(x)[f] \, y  \hbox{  $~\forall f \in \Cinfty{M} $, $x \in E_{-1}$ and $ y \in \Gamma (E_{-i})$,} $$
 for some vector bundle morphism $ \rho \colon E_{-1} \to TM $ called the \emph{anchor map}. 
\end{definition} 
 
 {
It is well-known \cite{Voronov,Poncin} that given a graded vector bundle $E_\bullet=\bigoplus_{i=1}^dE_{-i}$, Lie $\infty$-algebroid structures on $E$ are in 1-to-1 correspondence with NQ-structures over the N-manifold $\Gamma(\cdot, S^\bullet E_\bullet^*)$.  Moreover, by Batchelors theorem any smooth NQ-manifold is of that form, i.e. $ \sh{M}$ is (non-canonically) isomorphic to $\Gamma(\cdot,  S^\bullet E_\bullet^*)$ for some graded vector bundle $E_\bullet$.
}\\

\begin{convention}
{From now on, we will write ``Lie $\infty$-algebroid'' instead of ``negatively graded Lie $\infty$-algebroid''. }
\end{convention}
{We warn the reader, that general (not necessarily negatively graded) Lie $\infty$-algebroids are much more involved: Its higher brackets may not be $\Cinfty{M}$-linear, but they may have many anchor maps, see e.g. \cite{referee}.}

Given a  Lie $\infty$-algebroid structure over $M$ as above, $ {\fol{F}} := \rho (\Gamma(E_{-1}))$ is easily shown to be a singular foliation, called the \emph{basic singular foliation}. Also, $l_1$ is $\Cinfty{M}$-linear and
$$  \dots \stackrel{l_1}{\to} E_{-3}  \stackrel{l_1}{\to} E_{-2} \stackrel{l_1}{\to}  E_{-1} \stackrel{\rho}{\to}  TM $$
is a complex of vector bundles, called the \emph{linear part} of $E_\bullet$.
\begin{remark}
\normalfont
{
An NQ-manifold also induces a singular foliation on $M$, called basic singular foliation, defined as $\rho(\Gamma(E_{-1}))$ for any Lie $\infty$-algebroid $E_\bullet$ as above. In particular, the base manifold of an NQ-manifold is partitioned into leaves, called $\NQ{M}$-leaves. Its leaf space will be denoted by $M/ \NQ{M}$.
}
\end{remark}

\begin{example}
\normalfont
For a manifold $\Sigma$, a Lie $\infty$-algebroid of the NQ-manifold $T[1]\Sigma$ is the tangent Lie algebroid $T\Sigma$. 
\end{example}

Let $\NQ{M} =(M, \sh{\NQ{M}},Q_\NQ{M}) $ and $\NQ{M}'=(M', \sh{\NQ{M}'},Q_\NQ{M'})$ be NQ-manifolds. An \emph{NQ-manifold morphism} $\Phi:  \NQ{M} \to \NQ{M}'$ is by definition a graded commutative differential algebra morphism $\Phi^*:\sh{\NQ{M}'} \to \sh{\NQ{M}}$. 

\begin{example}
\normalfont {
A smooth map $\phi:\Sigma\to \Sigma'$ induces an NQ-morphism $T[1]\Sigma\to T[1]\Sigma'$ by the pullback $\phi^*:\Omega^\bullet_{\Sigma'}\to \Omega^\bullet_\Sigma$. }
\end{example}

When $\sh{\NQ{M}}=\Gamma(\cdot, S^\bullet(E^*_\bullet))$ and $\sh{\NQ{M}'}=\Gamma(\cdot, S^\bullet((E')^*_\bullet))$, the latter is entirely described by:
\begin{enumerate}
	\item[$\bullet$] a smooth map $\phi: M \to M'$ called the \emph{base map} of $\Phi$;
	\item[$\bullet$] a sequence $(\phi_n)_{n \geq 1}$ of degree zero vector bundle morphisms $\phi_n \colon S^n (E) \to E' $ over $\phi$ called the \emph{Taylor coefficients} of $ \Phi$.
\end{enumerate}	 

For instance, an NQ-manifold morphism $ T[1]I\times  \NQ{M} \to \NQ{M}' $ is a differential graded algebra morphism\footnote{The symbol $\tilde{\otimes}$ stands for the completed tensor product, which is the right operation to deal with products of NQ-manifolds}: 
$$ \Phi \colon (\sh{\NQ{M}'},Q_{\NQ{M}'}) \longrightarrow  (\sh{\NQ{M}},Q_{\NQ{M}}) \tilde \otimes \Omega^\bullet_I,~~Q_{\NQ{M}} + {\rm d}_{dR}).$$
For any $t \in I$, $ \Phi$ admits a restriction to  an NQ-manifold morphism  $\Phi_t\colon \NQ{M} \hookrightarrow \{t\} \times  \NQ{M} \to \NQ{M}' $. 
The NQ-morphism $\Phi \colon T[1]I \times \NQ{M} \to \NQ{M}' $ is said to be \emph{constant} if the following diagram commutes:
$$ \xymatrix{  T[1]I\times  \NQ{M}  \ar[d]^{\Phi} \ar[r]^{pr} & \NQ{M} \ar[dl]^{\Phi_0} \\ \NQ{M}'& } $$
\begin{remark}
\normalfont
For a constant  NQ-manifold morphism $  T[1]I\times  \NQ{M}  \to \NQ{M}' $, the restrictions $ \Phi_t$ do not depend on $t \in I$. However $\Phi_t=\Phi_0$ for all $t\in I$ does not imply that $\Phi$ is constant.
\end{remark}

\begin{definition}
A \emph{homotopy between two NQ-manifold morphisms} $\Phi_0, \Phi_1\colon \NQ{M} \to \NQ{M}'$ is an NQ-manifold morphism
 $ T[1]I\times  \NQ{M}  \to \NQ{M}'$  whose restrictions to the endpoints of $I=[0,1]$ are $\Phi_0$ and $ \Phi_1 $ respectively.
\end{definition} 
A homotopy $T[1]I \times  \NQ{M}  \to \NQ{M}' $  between NQ-manifold morphisms $\Phi_0$ and $ \Phi_1 $  constant after restriction to $ {T[1]}[1-\epsilon,1]\times  \NQ{M}  \to \NQ{M}'  $ and a homotopy $[0,1]\times  \NQ{M}  \to \NQ{M}' $  between NQ-manifold morphisms $\Phi_1$ and $ \Phi_2 $  constant after restriction to $ {T[1]}[0, \eta] \times  \NQ{M} \to \NQ{M}'  $ can be concatenated so that the concatenation remains smooth.
Also, any homotopy can be transformed by a smooth rescaling of $I$ to a homotopy with these properties.
\\

Consider an N-manifold  with functions $\sh{}$ over $M$. Every degree zero vector field $Y\in Der(\sh{})$ induces a derivation of $\Cinfty{M}=\mathcal O_0$, \emph{i.e.~}a vector field $\underline{Y} $ on the base manifold $M$.

\begin{lemma}\label{lem:flow}
Let $Y$ be a degree zero vector field on an NQ-manifold $\NQ{M}=(M, \sh{\NQ{M}},Q_\NQ{M})$ over a manifold $M$.
\begin{enumerate}
    \item For all  fixed $t \in \mathbb R$, the vector field $Y$ admits a time-$t$ flow $\Phi_t^Y\colon \NQ{M} \to  \NQ{M}$ if and only if the induced vector field $\underline{Y} $ on $M$ admits a time-$t$ flow.
    \item Assume $Y$ is closed, \emph{i.e.}~$[Y,Q]=0$. Then, for any admissible $t$, the flow $\Phi_t^Y:\NQ{M}\to \NQ{M} $ is an NQ-morphism. 
    \item Assume $Y$ is exact, \emph{i.e.~}there exists $X$ such that $[X,Q]=Y$.  Then there exists an NQ-morphism
$$\Phi^{X} : T[1]\mathbb R\times \NQ{M} \to \NQ{M}, $$ 
maybe defined only in a neighborhood of $\{0\} \times \NQ{M} $,  such that for all admissible $t\in \mathbb R$, the restriction $\Phi^X_t:\NQ{M}\to \NQ{M}$ is the flow of $Y$ at time $t$. In particular all $\Phi_t^Y$ are homotopic NQ-morphisms. 
\end{enumerate}
\end{lemma}
\begin{proof}
 For the two first points, see Chapter 5 in \cite{MR2102797} (that deals with super-manifolds, instead of graded manifolds, but the arguments can be repeated word by word).
    Let us prove the third item. Consider the following map  \footnote{As in Section 3.4.4 in \cite{LLS}, we implicitly consider elements of degree $k$ in $\sh{\NQ{M}} \tilde{\otimes} \Omega^\bullet_I $ as being elements of the form $ F_t \otimes 1 + G_t \otimes dt $ with $F_t,G_t \in \sh{}$ being elements of degree $k$ and $k-1$ respectively that depend smoothly on a parameter $t \in I$.}:
    \begin{align} \label{eq:intertwiner}\begin{array}{rcl} \sh{\NQ{M}}&\to& \sh{\NQ{M}} \, \tilde{\otimes}  \, \Omega(I) \\   F & \mapsto & \Phi_t^Y(F) \otimes 1 + X \circ \Phi_t^Y  (F)\otimes dt \end{array}
    \end{align}
    It is easily checked to be a graded algebra morphism.
    Furthermore, the differential equation 
    \begin{align*}
         \frac{\partial \Phi^Y_t }{\partial t}=Y\circ \Phi^Y_t
         =[X,Q] \circ  \Phi^Y_t=(X\circ  \Phi^Y_t)\circ Q + Q\circ (X\circ  \Phi^Y_t)
    \end{align*}
    implies that \eqref{eq:intertwiner} intertwines $Q$ and $Q+\rm d_{dR}$. This completes the proof.
\end{proof}

\subsubsection{Definition of homotopy groups of NQ-manifolds}

We summarize in this section several ideas coming from \v{S}evera, see \cite{1707.00265} and \cite{SeveraIntegration}. \\

 {Let $ \NQ{M} = (M, \sh{\NQ{M}},Q_\NQ{M})$ be an NQ-manifold.  Let  $ E_\bullet$ be a graded vector bundle over $M$ such that there exists an isomorphism $\sh{\NQ{M}}\cong \Gamma(\cdot, S^\bullet E_\bullet^*)$. For a manifold $\Sigma$,} an NQ-morphism $\Phi:T[1]\Sigma\to \NQ{M}$ is entirely described by

\begin{equation}  \Phi: (\NQ{M}, Q) \mapsto (\Omega^\bullet_\Sigma,{\rm d}_{dR}).\end{equation}
The latter is entirely described by:
\begin{enumerate}
	\item[$\bullet$] its base map, which is a smooth map $\phi: \Sigma \to M$
	\item[$\bullet$] its Taylor coefficients $\phi_n \colon \wedge^n T\Sigma \to E_{-n} $.
\end{enumerate}	 
 {such that the induced map $(\Gamma(S^\bullet E_\bullet^*), Q_{\NQ{M}}\to (\Omega^\bullet_\Sigma, \rm d_{dR}) $ is a chain map.}
\begin{example}
\label{ex:constant}
\normalfont {
A map $T[1]\Sigma\to \mathcal M$ is called constant, if it factors through the point NQ-manifold $(*, \mathbb R, 0)$. This means that the base map is constantly some point $m\in M$ and all Taylor coefficients are equal to zero. Such a constant map will be denoted by $\underline{m}$.}
\end{example}

Let us now define the homotopy and fundamental groups based at a point $m\in M$.
  Let $N$ be the north pole of the sphere $S^n$. Consider the set of homotopy classes of maps from $T[1]S^n$ to  $\NQ{M}$ whose restriction to a neighborhood of the north pole is constantly $ m$. As for usual manifolds, this set has a group structure, and this group is Abelian for all $n \geq 2$. For $ n \geq 2$, it is referred to as the \emph{$n$-th homotopy group} based at $m$ and denoted by $\pi_n(\NQ{M},m) $. For $n=1$, we call it the \emph{fundamental group of $(\NQ{M}, Q) $} based at $m \in M$. 
  Consider now the set of homotopy classes of \emph{paths}, (\emph{i.e.}~NQ-morphisms from $T[1]I$ to $ \NQ{M}$) valued in $\NQ{M}$ which are constant near $0$ and $1$. This set comes equipped with a natural groupoid structure over $M$, referred to as the \emph{fundamental groupoid of $\NQ{M} $} and denoted by $\mathbf \Pi(\NQ{M}) \rightrightarrows M$.  {When $M$ is a manifold, the NQ-manifold homotopy groups and fundamental groupoid of $T[1]M$ coincide with the classical homotopy groups and fundamental groupoid of the manifold $M$.}

\begin{proposition}
Let $\NQ{M}$	 and  $ \NQ{M}'$ be
two NQ-manifolds.
\begin{enumerate}
    \item An NQ-morphism $\Phi \colon   \NQ{M} \to \NQ{M}'$ with base morphism $ \phi$ induces a group morphism $\pi_n( \NQ{M},m) \to \pi_n( \NQ{M}',\phi(m)) $ for all $ m \in M, n \in \mathbb N$ and a groupoid morphism $ \mathbf \Pi(\NQ{M}) \to  \mathbf \Pi(\NQ{M}')$. 
    \item Two homotopic NQ-morphisms $\Phi_0,\Phi_1$ over the same base map $\phi$ induce the same group and groupoid morphisms.
\end{enumerate}
\end{proposition}
\begin{proof}
We prove it for homotopy groups: the proof for the fundamental groupoid is similar.
The group and groupoid morphisms in the first item are simply induced by the push-forward:
 $$ \begin{array}{rcl} \{ \hbox{Map}( T[1]S^n, \NQ{M} ) \}& \longrightarrow & \{ \hbox{Map}( T[1]S^n, \NQ{M}') \}, \\ \sigma &  \mapsto & \Phi \circ \sigma \end{array} $$
 which transforms a map constantly $m$  to a map constantly $\phi(m)$ in a neighborhood of the north pole. 
 For the second item, it suffices to check that for $\Phi \colon T[1]I \times \NQ{M} \to \NQ{M}' $ a homotopy between $ \Phi_0$ and $\Phi_1$, the composition $  \Phi \circ ({\mathrm{id}} \times \sigma)\colon T[1]I \times T[1]S^n \to \NQ{M}'  $ is for every  $ \sigma \in \hbox{Map}( T[1]S^n, \NQ{M} )  $ a homotopy relating $ \Phi_0 \circ \sigma$ and $ \Phi_1 \circ \sigma$.
\end{proof}

Here is an important consequence of this proposition.

\begin{corollary}
\label{cor:homotopyGroups}
The homotopy groups, fundamental groups and fundamental groupoids of any two homotopy equivalent NQ-manifolds are isomorphic.	
\end{corollary}

The following lemma will also be of interest:

\begin{lemma}
\label{lem:restrLeaf}
Let $ L$ be the leaf through $m$ of an NQ-manifold $\NQ{M}$.
Then the restriction $ {\mathfrak i}_L \NQ{M}$ of $ \NQ{M} $ to $L$ is an NQ-manifold and $ \pi_n(\NQ{M},m) = \pi_n({\mathfrak i}_L\NQ{M},m) $. Also, the {orbits} of the fundamental groupoid $ \mathbf \Pi (\NQ{M})$ are the leaves of the basic singular foliation
of $\NQ{M} $.
\end{lemma}
\begin{proof}
This first part of the lemma follows obviously from the fact that the base map of an NQ-manifold morphism from a connected manifold to $ \NQ{M}$ can not ``jump'' from one leaf to an other leaf, and has to be valued in a given $\NQ{M}$-leaf. For the second part, assume $\sh{\NQ{M}}=\Gamma(\cdot, S^\bullet E_\bullet^*)$. By the definition of leaves, there exists a path $\gamma:I\to M$ and a path $e:I\to E_{-1}$, such that $\dot\gamma(t)=\rho(e(t))$. The pair $(\gamma, e)$ is made of the base map and of the Taylor coefficient of an NQ-manifold morphism $T[1]I\to \mathcal M$.
\end{proof}

Altogether,  homotopy groups $ (\pi_n(\NQ{M}),m)_{m \in M}$ of an NQ-manifold $ \NQ{M}$ with base manifold $M$ form a bundle of groups over $M$, that we call the \emph{$n$-th homotopy group bundle}. 

\begin{definition}
A section $ \sigma : m \mapsto \sigma(m) \in \pi_n(\NQ{M}, m) $ 
 of the $n$-th homotopy group bundle is said to be \emph{smooth} there exists an NQ-morphism  $M \times S^n \to \NQ{M}$ whose restriction to $ \{m\} \times S^n$ is a representative of $ \sigma(m) $ for all $m \in M$. 
  We denote these smooth sections by $\Gamma(\pi_n(\NQ{M})) $. 
 \end{definition}
 
For the fundamental groupoid, smooth sections of the source map, and \emph{smooth bisections}, are defined in the same way.\\

Let us spell out the special case $n=0$. We define $\pi_0(\NQ{M},m)=M/\NQ{M}$ to be the pointed set $M/\NQ{M} $  of $\NQ{M}$-orbits, the orbit of $m$ being the marked point. Also: 

\begin{definition}
\label{def:pi0}
We define $\Gamma(\pi_0(\NQ{M})) $ to be the group ${\mathrm{Diff}}(M/\NQ{M}) $ of bijections of the orbit space $M/\NQ{M} $ that come from a smooth diffeomorphism of $M$ mapping every $\NQ{M}$-leaf to an $\NQ{M}$-leaf. In particular,  $\Gamma(\pi_0(\NQ{M}))$ has a natural group structure.
\end{definition}

\subsection{A long exact sequence for NQ-manifold fibrations }

\subsubsection{NQ-manifold fibrations and Ehresmann connections}

Throughout this section:
\begin{enumerate}
    \item $(E_\bullet ,\{l_k\}_{k\geq 1},\rho_E)$ is a Lie $\infty$-algebroid with associated NQ-manifold $ \NQ{M}= (M, \sh{\NQ{M}},Q_\NQ{M})$.
    \item $B \to L$ is a Lie algebroid with anchor $\rho_B:B\to TL$, seen as an NQ-manifold $B[1]$ when equipped with functions $ \Gamma(\wedge^\bullet B^*)$ and Chevalley-Eilenberg differential (see \cite{MR2157566}, Chapter 7). 
    \item  $ P : \NQ{M} \to B[1]$ is an NQ-manifold morphism. 
For degree reasons, all its Taylor coefficients are zero except for the first one, so that $P$ is entirely described by a vector bundle morphism denoted by $P$:
\begin{equation}
\label{eq:shortExact} \xymatrix{ E_{-1}\ar[r]^{P} \ar[d] & B  \ar[d] \\  M \ar[r]^p & L} 
\end{equation}
such that  $ \wedge^\bullet P^* \colon \Gamma( \wedge^\bullet B^*  )\to \Gamma( \wedge^\bullet E^*_{-1} ) \subset \Gamma(S^\bullet(E^*_\bullet)) $  is a differential graded commutative algebra morphism.
\end{enumerate} 

\begin{definition}
\label{def:fibrations} 
When $P \colon E_{-1} \to B$ is a surjective submersion, \emph{i.e.~}the base map $ p\colon M \to L$ is a surjective submersion and $P_m : E_{-1}|_{m} \to  B|_{p(m)} $ is a surjective linear map for every $m \in M$, we say that $P$ is an
NQ-manifold fibration over a Lie algebroid.
\end{definition}

\begin{example}
\normalfont
When $\NQ{M}=T[1]M $ and $B=TL$ are manifolds,  we recover usual manifold fibrations. When both $\NQ{M}=\mathfrak g[1] $ and $B=\mathfrak h$ (i.e. $L=*$) are Lie algebras, we recover Lie algebra epimorphisms.
\end{example}

Let us define the fibers of an NQ-manifold fibration. 
Consider the vector bundles over $M$ defined by $ K_{-i} = E_{-i}$ for all $i \geq 2$ and $K_{-1} := {\mathrm{Ker}}(P) $.
The Lie $\infty$-algebroid brackets of $E$ restrict to sections of $ K_\bullet \to M$, and therefore equip the latter with a Lie $\infty$-algebroid structure.
We denote by $ {\NQ{T}}$ its associated NQ-manifold. Since the anchor  $ {\NQ{T}}$ is by construction valued in the fibers of $p \colon M \to L$, it restricts to an NQ-manifold $ {\NQ{T}}_\ell$ on the fiber $p^{-1}(\ell)  $ over any $\ell \in L$.
We call the NQ-manifold $ {\NQ{T}}_\ell$ the \emph{fiber over $\ell \in L$}.

\begin{example}
\normalfont
When both $\NQ{M} $ and  $B$ are manifolds, we recover usual  fibers of a manifold fibration.
When both $\NQ{M} $ and $B$ are Lie algebras, the fiber is simply the kernel of the Lie algebra epimorphism $P$.
\end{example}

\begin{example}
\normalfont
Lie algebroid actions on $N$-manifolds, as defined in \cite{BrahicZambon}, provide a wide class of examples of Lie algebroid fibrations.
\end{example}

\begin{definition}
An Ehresmann connection $(E_\bullet, H)$ for an NQ-manifold fibration over a Lie algebroid $ P \colon \NQ{M} \to B$ is  { a choice of $E_\bullet$ such that $\sh{\NQ{M}}=\Gamma(\cdot, S^\bullet E_\bullet^*)$ and} a sub-bundle $ H \hookrightarrow E_{-1} $ in direct sum with $K_{-1}={\mathrm{Ker}}(P) \subset E_{-1}$.
\end{definition}

An Ehresmann connection induces a \emph{horizontal lift} $ \tilde{H}\colon \Gamma(B) \to \Gamma(E_{-1})$: A section $b \in \Gamma(B)$ is lifted to the unique section of $H$ which is $P$-related to $b$. 
For every section $ b \in \Gamma (B)$ the vector field $ \rho_{\NQ{M}} \circ \tilde{H} \, (b)$ is $p$-related to the vector field $ \rho_B (b)$.

\begin{definition}
\label{def:Ehresmann}
An Ehresmann connection for an  NQ-manifold fibration over a Lie algebroid $ P \colon \NQ{M} \to B[1]$  as above  is said to be \emph{complete} if for all $ b \in \Gamma(B) $ the integral curve of the vector field $ \rho_\NQ{M} \circ \tilde{H} (b) $ starting at $m$ is defined if and only if  the integral curve of the vector field $ \rho_B (b)$ starting at $ p(m)$ is defined.
\end{definition}

\begin{remark}
\normalfont
Of course, every NQ-manifold fibration over a Lie algebroid  admits Ehresmann connections, but there may not exist complete Ehresmann connections. For instance when $B$ is transitive and the fibers $ (\NQ{T}_{\ell})_{\ell \in L}$ over two points $\ell,\ell' \in L$ are not diffeomorphic, Proposition \ref{prop:fibersAreIso} below obstructs the existence of complete Ehresmann connections.
\end{remark}

\begin{example}
\normalfont
When $\NQ{M} $ is  a manifold $M$, we recover usual (complete) Ehresmann connections for manifold fibrations.
When both $\NQ{M}=\mathfrak g[1]$ and $B=\mathfrak h$ are Lie algebras, a connection is a linear section of the epimorphism $P$ and every Ehresmann connection is complete.
\end{example}

\begin{remark}
\normalfont
When $\NQ{M}$ is a Lie algebroid, our objects match those introduced in  \cite{MR3897481} and \cite{Zhu}, but with several differences in vocabulary. NQ-manifold fibrations do not coincide with Lie algebroid fibrations in \cite{Zhu} since they assume that the Lie algebroid fibration comes with a complete Ehresmann connection. The correspondence is ``Lie algebroid fibration in \cite{Zhu}''= ``NQ-manifold fibration in our sense with $\NQ{M}$ a Lie algebroid'' + ``a complete Ehresmann connection''.
Also,  ``Ehresmann connections'' in \cite{Zhu}  are always complete, and therefore correspond to our ``complete Ehresmann connections''.
When $\NQ{M}=\mathfrak g[1]$ is a Lie algebroid and $B$ is transitive, our NQ-manifold fibrations are the Lie algebroid submersions of \cite{MR3897481} and ``complete Ehresmann connections'' have the same meaning.
\end{remark}

Sections of $\Gamma(B) $ (resp. $E_{-1}$) can be considered as degree $-1$ vector fields on the NQ-manifold $B[1]$ 
(resp. $\NQ{M}$): it suffices to let $ b \in \Gamma(B) $ (resp. $e \in \Gamma(E_{-1})$) act by contraction $ {\mathfrak i}_b$ (resp. $ {\mathfrak i}_e$) on their respective graded algebras of functions. 
The degree zero vector fields $ [ {\mathfrak i}_{\tilde{H}(b)}, Q_{\NQ{M}}]$ and  $[ {\mathfrak i}_{b}, {\mathrm d}_{B}] $ are $P$-related for every Ehresmann connection $H$. If $H$ is complete, then their flows are also related:
 
\begin{proposition}
\label{prop:aboutComplete}
Consider an Ehresmann connection $(E_\bullet, H)$ for an NQ-manifold fibration $ P \colon \NQ{M} \to B[1]$ over a Lie algebroid. The following are equivalent:
\begin{enumerate}
    \item The Ehresmann connection $(E_\bullet, H)$ is complete. 
    \item For all $ b \in \Gamma(B)$, the time $t$ flow of the degree $0$ vector field $ [ {\mathfrak i}_{\tilde{H}(b)}, Q]$ is defined if and only if the time $t$ flow of the degree $0$ vector field $[ {\mathfrak i}_{b}, {\mathrm d}_{B}] $ is defined.
\end{enumerate}
\end{proposition}
\begin{proof}
By the first item in Lemma \ref{lem:flow}, the flow of a degree $0$ vector field exists if and only if the flow of its induced vector field on its base manifold exists. 
In our case, the induced vector fields of  $ [ {\mathfrak i}_{\tilde{H}(b)}, Q_{\NQ{M}}]$ and $[ {\mathfrak i}_{b}, {\mathrm d}_{B}] $ are 
$ \rho_\NQ{M} \circ \tilde{H} (b) $ and $ \rho_B (b) $ respectively. Hence the equivalence of both items is a direct consequence of Definition \ref{def:Ehresmann}.
\end{proof}

In differential geometry, complete Ehresmann connections allow to identify fibers with each other: the same occurs in the present context for transitive Lie algebroids. 

\begin{proposition}
\label{prop:fibersAreIso}
If a complete Ehresmann connection $(E_\bullet, H)$ for an NQ-manifold fibration over a transitive Lie algebroid $ P \colon \NQ{M} \to B[1]$ exists, then the fibers $ {\NQ{T}}_{\ell}$ and  $ {\NQ{T}}_{\ell'}$ over any two $\ell, \ell' \in L$ are diffeomorphic NQ-manifolds. 
\end{proposition}

\begin{proof}
There exists a vector field $X $ on $L$ whose time $1$ flow exists globally and maps $\ell$ and $\ell'$. Since $B$ is transitive, there exists a section $b \in \Gamma(B)$ with $ \rho_B(b)=X$.  Since the base vector field of $ [ {\mathfrak i}_{\tilde{H}(b)}, Q_{\NQ{M}}] $ is $p$-related to $X$, and since the time-$1$ flow  of  $\rho_B(b)=X$ is well-defined, Proposition \ref{prop:aboutComplete} implies that the time-$1$ flow $\Phi_1$ of $ [ {\mathfrak i}_{\tilde{H}(b)}, Q_{\NQ{M}}] $ is a well-defined NQ-manifold diffeomorphism.   Since the base  diffeomorphism $ \phi_1 \colon M \to M$  of $\Phi_1$ is over the time $1$- flow of $ X$, it maps the fiber $ p^{-1}(\ell)$ to the fiber $ p^{-1}(\ell')$. Therefore, the restriction of $ \Phi_1$ to  $ {\NQ{T}}_{\ell}$ is the desired NQ-manifold diffeomorphism ${\NQ{T}}_{\ell}\overset{\cong}{\to} {\NQ{T}}_{\ell'}$. 
\end{proof}
\begin{remark}
\normalfont
Proposition \ref{prop:fibersAreIso} obviously extends to the non-transitive case as follows:  $ {\NQ{T}}_{\ell}$ and  $ {\NQ{T}}_{\ell'}$ over any two $\ell, \ell' \in L$ in the same $B$-leaf are diffeomorphic NQ-manifolds.
 When $\NQ{M}$ {comes from} a Lie algebroid, this follows from Theorem C in \cite{MR3897481}. 
\end{remark}

\subsubsection{Horizontal lifts for complete Ehresmann connections}

The equivalent of the following lemma for fibrations of ordinary manifolds is well-known under the name of ``homotopy lifting property'', see \cite{MR1325242}, Chapter 7, or \cite{MR1867354} Section 4.2. Its Lie algebroid equivalent is implicit in the proof of Theorem 1.4 in \cite{Zhu}.
\begin{flemma}
	\label{lem:lifts}
	Let  $ P \colon \NQ{M} \to B[1]$ be an NQ-manifold fibration over a Lie algebroid as in Definition \ref{def:fibrations} that admits a complete Ehresmann connection $(E_\bullet, H)$. Given
	\begin{enumerate}
	    \item[a)]  a manifold $\Sigma$ called \emph{parameter space};
	    \item[b)] an NQ-morphism $e\colon T[1]\Sigma \mapsto \NQ{M} $ called \emph{initial shape};
	    \item[c)] a Lie algebroid morphism $b: T[1](I \times \Sigma) \to B[1] $ called \emph{base map}
	\end{enumerate} 
	making the following diagram of NQ-manifold morphisms  commutative
			\begin{equation}
	\label{eq:lemmaLift}
		\xymatrix{ 
			 T[1]\Sigma\ar[rr]^e \ar@{_(->}[d]_{u \to (0,u) }&   &   \NQ{M}\ar@{->>}[d]^P \\
			T[1](I\times \Sigma) \ar[rr]_b&  & B[1]}
		\end{equation}
		there exists a natural NQ-manifold morphism ${\mathcal L}^H(e,b)\colon T[1](I\times \Sigma)--\to \NQ{M}$ called \emph{horizontal lift of the base map $b$ with initial shape $e$} such that the following diagram commutes: 
	\begin{equation}
	\label{eq:lemmaLift2}
		\xymatrix{ 
			 T[1]\Sigma\ar[rr]^e \ar@{_(->}[d]_{u \to (0,u) }&   &   \NQ{M}\ar@{->>}[d]^P \\
			T[1](I\times \Sigma) \ar[rr]_b  \ar@{-->}[urr]_{{\mathcal L}^H(e,b)}&  & B[1]}
		\end{equation}
\end{flemma}

\begin{proof}
Consider the following Cartesian diagram of vector bundle morphisms 
\begin{equation}
\label{diag:cartesian}\xymatrix{
(TI\times T\Sigma)\times_{b,B,P}E_{-1}\ar[rrr]\ar[dr] \ar[ddd]& & & TI\times T\Sigma\ar[ddd] \ar[dl]\\ 
&(I\times \Sigma)\times_{L}M\ar[r]\ar[d]&I\times \Sigma\ar[d]& \\
& M\ar[r]& L & \\
E_{-1}\ar[ur] \ar[rrr]&  & & B\ar[ul]
} \end{equation}

Given a horizontal distribution $D$ relative to $P:E_{-1}\twoheadrightarrow B$, the section $\frac{\partial}{\partial t}$ of $TI\times T\Sigma$ lifts to a unique section $X$ of $(TI\times T\Sigma)\times_{b,B,P}E_{-1}$, such that for all $(t,\sigma,m)\in (I\times \Sigma)\times_{L}M$, $X(t,\sigma,m)=(\frac{\partial}{\partial t},0,u)$, with $u\in D$. By construction $P(u)=b_{t,\sigma}(\frac{\partial}{\partial t},0)$.\\

We now prove the first item. We define ${\mathcal L}^H(e,b)$ as the composition of three NQ-morphisms:
    \begin{align}
    \label{eq:composition}
        T[1]I\times T[1]\Sigma\overset{\alpha}{\longrightarrow} 
        T[1]I\times (T[1]I\times  T[1]\Sigma) \times_{b,B[1],P} \NQ{M}\overset{\beta}{\longrightarrow}
          (T[1]I\times  T[1]\Sigma) \times_{b,B[1],P} \NQ{M} \overset{\gamma}{\to} \NQ{M}
    \end{align}
    which we now describe. $ (T[1]I\times  T[1]\Sigma) \times_{b,B[1],P} \NQ{M} $  stands for fiber products of NQ-manifolds, which make sense because $P$ is a surjective submersion. Its base manifold is $ (I\times \Sigma)\times_{L}M $ and its associated vector bundle is given in degree $-1$  by $(TI \times T\Sigma) \times_{b,B,P} E_{-1}$.
        In particular, the section $X$ defined above can be seen as a degree minus one vector field $\mathfrak i_X$ on $ (T[1]I\times  T[1]\Sigma) \times_{b,B[1],P} \NQ{M} $. 
    \begin{enumerate}
        \item[$(\alpha)$] The first arrow  is the composition of $T[1]I\times T[1]\Sigma\to T[1]I\times T[1]I\times T[1]\Sigma, (t,\sigma)\mapsto (t,0,\sigma)$ with $\mathrm{id}_{T[1]I}\times A$, where $A: T[1]I\times T[1]\Sigma\to (T[1]I\times T[1]\Sigma)\times_{B[1]} \NQ{M}$ is the natural map associated to the pullback.
        \item[$(\beta)$] The central arrow is the flow $NQ$-morphism of the degree zero vector field $[\tilde Q, \mathfrak i_X]$, as constructed in Lemma \ref{lem:flow}, where $\tilde Q$ is the derivation of $(T[1]I\times  T[1]\Sigma) \times_{b,B[1],P} \NQ{M}$.
        \item[$(\gamma)$] The last arrow is the projection to $\NQ{M}$.
    \end{enumerate}
    Let us check that $\mathcal L^H(e,b)$ satisfies the required conditions.
Let us consider its restriction to $ \{0\} \times  T[1]\Sigma$. The morphism $(\beta)$ in   
\eqref{eq:composition}, restricted to $\{0\} \times  T[1]I \times T[1]\Sigma \times \NQ{M}$, is the projection onto the second component, so that the restriction of $(\beta)$ and $(\gamma)$ to  $\{0\} \times  T[1]I \times T[1] \Sigma \times \NQ{M}$ consists in projecting onto $ \NQ{M}$.
By definition of $(\alpha)$, the restriction to  $ \{0\} \times  T \Sigma$ coincides with $e$.

Let us prove that $P \circ \mathcal L^H(e,b) =  b $.
The following diagram is commutative,  where the upper horizontal line is \eqref{eq:composition}, the lower line is constructed as  \eqref{eq:composition}, with $B$ replacing $\NQ{M}$  and $ b_1 $ replacing $e$, and with all vertical lines being induced by $P \colon \NQ{M} \to B$:
 \begin{align}
    \label{eq:composition2}
        \xymatrix{T[1]I\times T[1]\Sigma \ar@{=}[d]\ar[r]^--{\alpha}&
        T[1]I\times (T[1]I\times  T[1]\Sigma) \times_{b,B[1],P} \NQ{M} \ar[d]^{\mathrm{id}\times \mathrm{id}\times P}\ar[r]^{\beta}& 
          (T[1]I\times  T[1]\Sigma) \times_{b,B[1],P} \NQ{M}\ar[d]^{\mathrm{id}\times P}\ar[r]^>>{\gamma} & \NQ{M}  \ar[d]^P \\T[1]I\times T[1]\Sigma\ar[r]
          &
        T[1]I\times (T[1]I\times  T[1]\Sigma) \times_{b,B[1],id} B[1] \ar[r]& 
          (T[1]I\times  T[1]\Sigma) \times_{b,B[1],id} B\ar[r] & B[1] 
          }
    \end{align}
   This commutativity is straightforward, except for the central square, for which it expresses that $X_{(t,\sigma,m)} $, with $p(m)=\sigma $, projects to $ (\frac{\partial }{\partial t}, b_{t,\sigma} (\frac{\partial}{\partial t}))$. The composition of the three lower arrows coincides with $b$.
\end{proof}

\begin{definition}
	Let  $ P \colon \NQ{M} \to B[1]$ be an NQ-manifold fibration over a Lie algebroid as in Definition \ref{def:fibrations} that admits a complete Ehresmann connection $(E_\bullet, H)$. Let $(\Sigma,e,b)$ be a parameter space, an initial shape, a base map as in Lemma \ref{lem:lifts}.
	We call \emph{parallel transport of $e$ with respect to $b$} the restriction to $ \{1\} \times \Sigma  $ of the horizontal lift ${\mathcal L}^H(e,b) $. We denote it by $ {\mathcal P}^H(e,b) \colon T[1]\Sigma \to \NQ{M}$.
\end{definition}

    The parallel transport behaves well with respect to restrictions to submanifolds:
    
    \begin{proposition}
    \label{prop:restrictions}
    Let  $ P \colon \NQ{M} \to B[1]$ be an NQ-manifold fibration over a Lie algebroid as in Definition \ref{def:fibrations} that admits a complete Ehresmann connection $(E_\bullet, H)$. Let $\Sigma,e,b$ be a parameter space, an initial shape, and a base map as in Lemma \ref{lem:lifts}. For every submanifold $ {\mathfrak i} : \Sigma' \hookrightarrow \Sigma$, the following diagram commutes:
    	\begin{equation}
	\label{eq:lemmaLiftRestr}
		\xymatrix{ 
		T[1]\Sigma'	\, \, \ar@{^(->}[rrrr] \ar[drr]_{\mathcal P^{H}(e|_{T[1]\Sigma'},b|_{T(I \times \Sigma')}) } & & & & T[1]\Sigma \ar[lld]^{\mathcal P^{H}(e,b) }\\  && \NQ{M} && }
		\end{equation}
    \end{proposition}

\begin{proof}
A close look at its construction shows that the horizontal lift behaves well with respect to restriction to a submanifold $\Sigma' \hookrightarrow \Sigma$, \emph{i.e.~}it satisfies $ {\mathcal L}^H(e,b) |_{T[1](I \times \Sigma')} = {\mathcal L}^H(e|_{T[1]\Sigma'}, b|_{T(I \times \Sigma')}) $ or, as a diagram:  
	\begin{equation}
	\label{eq:lemmaLiftRestr0}
		\xymatrix{ 
			 T[1]\Sigma\ar[d]_e  \, \, \ar@{^(->}[rrr]^{u \hookrightarrow (0,u) }& & &  T[1](I\times \Sigma) \ar[d]^{b} \ar@{-->}[dlll]_{{\mathcal L}^H(e,b)} \\				\NQ{M}\ar@{->>}[rrr]^P & & & B[1] \\
			   T[1]\Sigma'\, \, \ar[u]^{e|_{T[1]\Sigma'}} \ar@{^(->}[rrr]_{u \hookrightarrow (0,u) } \ar@/^4.0pc/@{_(->}[uu] & & &  T[1](I\times \Sigma')  \ar[u]_{b|_{T(I \times \Sigma')}}  \ar@{-->}[ulll]_{ {\mathcal L}^H(e|_{T[1]\Sigma'}, b|_{T(I \times \Sigma')}) } \ar@/_4.0pc/@{^(->}[uu]_{}}.
		\end{equation}
		This implies the commutativity of the diagram \eqref{eq:lemmaLiftRestr}.
		\end{proof}

{For a good understanding of the next lemma, recall that for a $\NQ{N}\subset \NQ{M}$ 
a sub-NQ-manifold, NQ-morphisms valued in $\NQ{N}$ may  be homotopic in $\NQ{M}$ without being homotopic in $\NQ{N}$, exactly as for ordinary manifolds.
}
\begin{lemma}\label{lem:hotint}
    Let  $ P \colon \NQ{M} \to B[1]$ be an NQ-manifold fibration over a Lie algebroid as in Definition \ref{def:fibrations} and $\ell \in L$ a point.
    Consider two NQ-morphisms $\Phi_i\colon T[1]\Sigma \to \NQ{T}_\ell $, $i=0,1$. Then the following items are equivalent:
    \begin{enumerate}
        \item[(i)] $\Phi_0$ and $\Phi_1$ are homotopic in $\NQ{T}_\ell$,
        \item[(ii)] $\Phi_0$ and $\Phi_1$ are homotopic in $\NQ{M}$ through a homotopy $h  \colon T[1]I \times T[1]\Sigma \to \NQ{M}$ such that $P \circ h \colon T[1]I \times T[1]\Sigma \to B[1]$ is homotopic to the constant map $\underline{\ell}$. 
    \end{enumerate}

\end{lemma}
\begin{proof}
The implication  (ii) $ \implies $ (i) is trivial. For the converse implication,  let $h_B:T[1](I\times I\times \Sigma)\to B[1]$ be a contracting homotopy for $P \circ h$. Consider its horizontal lift ${\mathcal L}^H (h , h_B) \colon T[1]I (\times I \times \Sigma) \to \NQ{M}$ with initial shape $h \colon T[1]I \times T[1]\Sigma \to \NQ{T}_\ell $ with respect to the base path $h_B:T[1](I\times I\times \Sigma) \to B[1]$. 
\begin{center}
		\begin{figure}[!ht]
			\caption{\label{fig:edge} The neighborhood $ U$ and the curve $ \gamma(t)$. }
			\begin{center}
			\includegraphics[clip=true, trim = 0mm  150mm 0mm 5mm,scale=0.3]{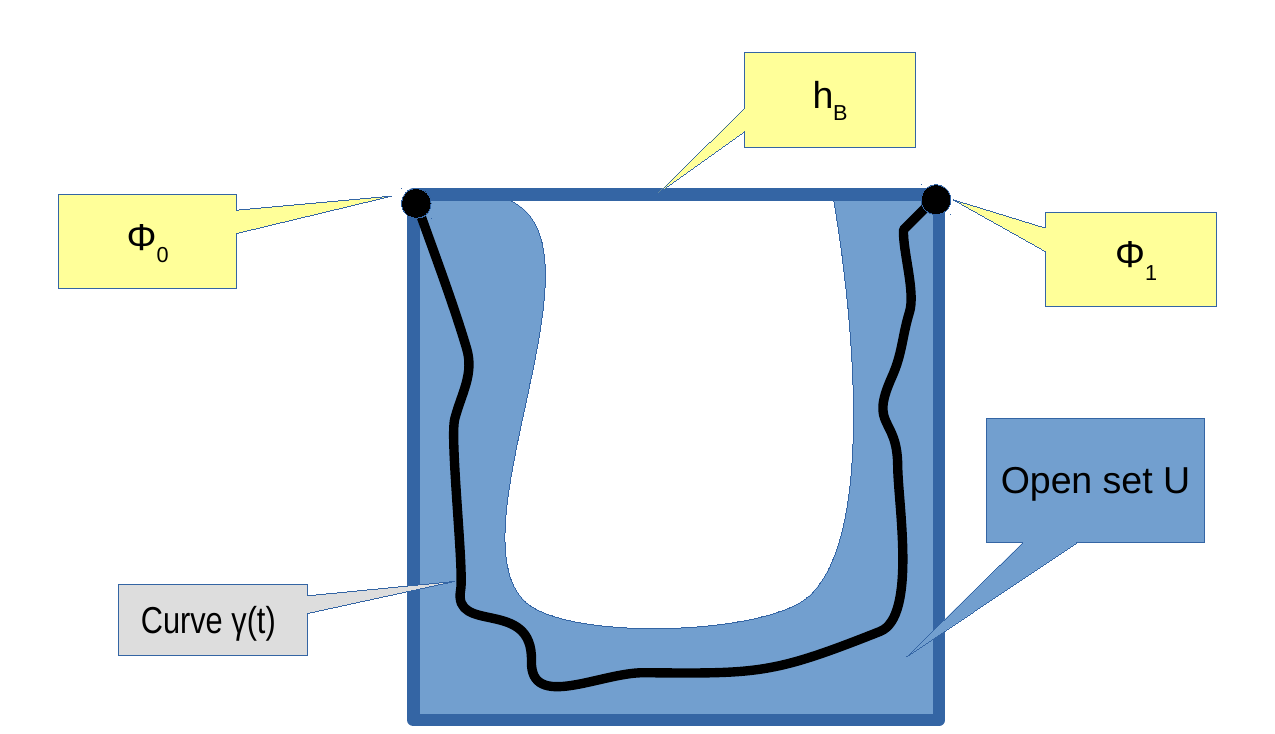}
			\end{center}
		\end{figure}
	\end{center}
Since $ h_B \colon  \colon T[1](I \times I) \to B[1]$ is constantly $ \ell$ on  $T[1](U \times \Sigma )$, with $U \subset I^2$  a neighborhood of the three edges of the square depicted in Figure \ref{fig:edge},
the restriction to $T[1](U \times \Sigma) $ of the horizontal lift ${\mathcal L}^H (h , h_B) $ is valued in $ {\NQ{T}}_\ell$. Let $\gamma : I \to I^2 $ a curve valued in $U$ relating $(0,1)  $ and $ (1,1)$ as in Figure \ref{fig:edge}. The restriction to  $ \{ \gamma(I) \} \times \Sigma $  of the  horizontal lift ${\mathcal L}^H (h , h_B) \colon T[1] (I \times I \times \Sigma) $ is a $ \NQ{T}_\ell$-valued homotopy between $ \Phi_0$ and $ \Phi_1$.

\end{proof}

\subsubsection{Main result}
We can now associate a long exact sequence to any NQ-manifold fibration over a Lie algebroid, generalizing Theorem 1.4 in \cite{Zhu}.

\begin{theorem}\label{thm:snake}
	Consider a fibration of NQ-manifold  over a  Lie algebroid $P \colon \NQ{M} \to B[1]$ as in Definition \ref{def:fibrations} that admits a complete Ehresmann connection. For  $\ell \in L$, denote by $ {\NQ{T}}_\ell  $ the fiber over $ \ell$.

Then there exists a family of group morphisms:
$$\begin{array}{rclll} \delta \colon & \pi_n (B[1],\ell ) & \longrightarrow & \Gamma ( \pi_{n-1}({\NQ{T}}_\ell ) )& \hbox{ for $ n \geq 2$ } \\
&  \pi_1 (B[1],\ell )&  \longrightarrow &  \mathrm{Diff}(M / \NQ{M}) &
\end{array}
$$
such that the sequence of groups morphisms
	$$
	\xymatrix{  & \dots\ar[r]^{P_*}  &  \ar[dll]_{\delta}\pi_{n+1}(B[1],\ell) \\
		\pi_n({\NQ{T}}_\ell,m)\ar[r]^{i_*}&\pi_n(\NQ{M},m)\ar[r]^{P_*}&\pi_n(B[1],\ell)\ar[dll]_{\delta}\\
		\pi_{n-1}({\NQ{T}}_\ell ,m)\ar[r]^{i_*}& \dots & \dots \\ 
}
		$$
		is exact for all $m \in p^{-1}(\ell)$.
		Moreover, $\delta(\pi_2(B[1], \ell)) $ 
		lies in the center of $\Pi_1(\NQ{T}_\ell)$,
		and, for a given  $\gamma \in \pi_1(B[1],\ell) $, the diffeomorphism $\delta (\gamma)$ of the orbit space of $ {\NQ{T}}_\ell$ fixes the  ${\NQ{T}}_\ell$-leaf of $ m$ if and only if $\gamma$ lies in  the image of  $ {P_*} \colon \pi_{1}(\NQ{M} ,m) \to \pi_{1}(B[1] ,\ell)$. 
\end{theorem}
\begin{proof} The fundamental Lemma \ref{lem:lifts} allows to make the proof of the present theorem similar to the usual proof for manifold fibrations as presented for instance in Theorem 4.41 in \cite{MR1867354} - with smooth maps being replaced by NQ-morphisms.
Although it follows the same structure as the proof of Theorem 1.4 in \cite{Zhu}, our point of view is more (NQ-) geometric.\\

\noindent
{\bf Construction of the connecting map $\delta$.}
Consider an NQ-morphism $ T[1]S^n\to B[1]$, which is constantly  $ \ell$  near the north pole $N \in S^n$. This last property allows to see it as an NQ-morphism $ \sigma \colon T[1](I\times S^{n-1})\to B[1]$ which is constantly $\ell$ near  $\{0 \} \times S^{n-1} $, $\{1\} \times S^{n-1} $ and $I \times \{N\}$.
For all $m \in p ^{-1}(\ell)$, we define $$ \delta|_m (\sigma) := {\mathcal P}^H(\sigma  , \underline{m} )$$ 
to be the parallel transport of the initial shape $\underline{m} \colon T[1]S^{n-1} \to \NQ{M}$ with respect to the base map $\sigma$.\\

Let us study this map. 
\begin{enumerate}
    \item  {\textbf{It is valued in the fiber $ {\NQ{T}}_\ell$.}} The horizontal lift of $e $ with respect to the base path $\sigma$ makes by definition the following diagram commutative:
	\begin{equation}
		    \label{eq:item1H}
		      \xymatrix{ 
		 	  T[1]S^{n-1} \ar[rr]^{\underline{m}} \ar@{_(->}[d]_{u \to (0 ,u)}  & & \NQ{M}\ar@{->>}[d]^P \\
		  T[1](I \times S^{n-1}) \ar[rr]_{\sigma } \ar@{-->}[urr]^{{\mathcal L}^H(\sigma  , \underline{m} )} & & B[1].}
    \end{equation}

    Since $\sigma $  is constantly $\ell  $ on $\{1\} \times S^n $, the commutativity of \eqref{eq:item1H} implies that $ \delta|_m ( \sigma) = {\mathcal P}^H(\sigma  , \underline{m} ) = {\mathcal L}^H(\sigma  , \underline{m} )|_{\{1\}\times T[1]S^{n-1}}$ is valued in the fiber $ {\NQ{T}}_\ell$.
    
    \item {\textbf{It is constantly equal to $m$ near the north pole}}. For every parameter manifold, parallel transport of a constant initial shape $\underline{m}$ with respect to a base map which is a constant map $\underline{\ell} $  is the constant map  $\underline{m} $ again. Since the north pole admits a neighborhood $U$ on which $\sigma|_ {T[1](I \times U)}$ is $\underline{\ell}$, it follows from Proposition \ref{prop:restrictions}, applied to the open submanifold $ \Sigma' =U $, that ${\mathcal P}^H(\sigma  , \underline{m} )$ is equal to $\underline{m}$ in a neighborhood of the north pole.
    
    \item {\textbf{ It preserves homotopy.}} Let $h\colon T[1] (I \times S^{n-1}\times I )\to B[1]$ be a homotopy between NQ-morphisms $\sigma_i\colon T[1](S^{n-1}\times I) \to B[1] $, $i=0,1$. Let ${\mathcal P} (\underline{m} , h) $ be the parallel transport of the initial shape $\underline{m}$ with respect to the base map $h$.
    By Proposition \ref{prop:restrictions}, the restrictions of ${\mathcal P} (\underline{m} , h) $  to the submanifolds $ T[1](\{i\} \times S^{n-1}\times I) \subset T[1]( I \times S^{n-1}\times I)$ is $\delta|_m (\sigma_i) $ for $i=0,1$. 
\end{enumerate}

Altogether these three points allow to define an induced group morphism:
 $$  \delta |_m(\sigma) \colon \pi_n(B[1] , \ell ) \longrightarrow  \pi_{n-1} ({\NQ{T}}_\ell ,m).$$
We have to check that $m \to \delta |_m(\sigma) $ is valued in smooth sections of the group bundle  $\pi_{n-1} ({\NQ{T}}_\ell)$.
Consider  \emph{(i)} the parameter space $ \Sigma = S^{n-1} \times p^{-1}(\ell)$, \emph{(ii)} the initial shape given by the projection onto $ pr_2 \colon S^{n-1} \times p^{-1}(\ell) \to \NQ{M}|_{p^{-1}(\ell)}$, and \emph{(iii)} the base map $\sigma \circ pr$ with $pr$ being the projection $I \times S^{n-1} \times  p^{-1}(\ell) \to I \times S^{n-1}$. The henceforth associated parallel transport is an NQ-morphism:
 $$ S^{n-1} \times  p^{-1}(\ell)  \longrightarrow {\NQ{T}}_\ell .$$ 
 Proposition \ref{prop:restrictions} implies that its restriction to the parameter submanifolds $S^{n-1} \times \{m\}$ is $\delta|_m(\sigma)$. This guarantees the smoothness of the section. \\

 \noindent {\bf The relation $\delta\circ P_*=0$.} 
		Consider an element of $ \pi_n(B,\ell )$ that can be represented by an NQ-morphism $ \sigma$  of the form $ \sigma = P_* (\tau)$ for some $ \tau \colon S^n \to \NQ{M} $ constantly equal to $m$ near $N$. 
		Let us compare $ \tau$ with the horizontal lift $\mathcal L^H(\sigma, \underline{m})$ of the initial condition $\underline{\sigma} $ with respect to the base $ \sigma$. 
		
		We can see $ \tau$ as NQ-morphism $\tau \colon T[1](I \times S^{n-1}) \to \NQ{M} $ constantly equal to $m$ in neighborhoods of $\{0\} \times T[1]S^{n-1}$ and $\{1\} \times T[1]S^{n-1}$. 
		This allows to consider the concatenation $\tau^{-1}*L^H(\sigma, \underline{m})$, where $\tau^{(-1)}$ is the inverse path of $\tau$, \emph{i.e.}~$ \tau \circ (1-t \times {\mathrm{id}}_{S^{n-1}}) $. Its projection through $P$ being the concatenation of $\sigma^{-1} $ with $\sigma $ is homotopic to $\underline \ell$. By Lemma \ref{lem:hotint}, the restrictions to the endpoints are homotopic in $\NQ{T}_\ell$.  These restrictions being $\underline{m}$ for both $\tau $ and $\delta|_{m}(\sigma) $, this means that $\delta(P_*([\tau]))$ is zero.  \\

		\noindent {\bf Exactness at $\pi_n(B[1],b_0)$.}  Let  $\sigma \colon T[1]S^n \to B[1] $. Consider $\mathcal L^H(\sigma,\underline{m}): T[1] (I \times S^{n-1})  \to \NQ{M} $ in diagram \eqref{eq:item1H}. Recall that  $ \delta(\sigma)$ is represented by the restriction of  $\mathcal L^H(\sigma,\underline{m})$ to  $T[1]( \{1\} \times S^{n-1})$.  If $\delta(\sigma) =0  $, then there exists
		a homotopy  $ h \colon  T[1]([1,2] \times  S^{n-1})  \to  {\NQ{T}}_\ell$ whose restrictions to $ T[1]({t} \times S^{n-1})$ is $\delta(f) $ for $t$ near $1$ and $ \underline{m}$
		for $t$ near $2$. The concatenation of  $\mathcal L^H(\sigma,\underline{m})$ and $h$ is smooth in view of the boundary assumptions and yields an NQ-morphism $ \tau \colon T[1]([0,2]) \to \NQ{M} $, defining an element in $ \pi_n( \NQ{M}, m) $, whose image through $ P$ is $\sigma$ by construction. \\

		\noindent {\bf Exactness at $\pi_n(\NQ{M}, e_0)$.} 
		Let $ \tau \colon T[1]S^n \to \NQ{M}$ be such  that $ P_* (\tau)=0$, so that there exists an NQ-morphism $ h :T[1]( I \times S^n)  \to B[1]   $ whose restrictions to $ T[1](\{0\} \times S^n) $ is $ \underline{\ell}$ and to  $T[1]( \{1\} \times S^n )$ is $ \underline{\ell}$ is $P_* (\tau)$. In view of the first item of Lemma \ref{lem:lifts}, there exists $\mathcal L^H(h, \tau) $ that makes the following diagram commutative:
		 	$$ \xymatrix{ T[1](\{0 \} \times S^{n} )\ar[rr]^{\tau} \ar[d] &&   \NQ{M}\ar@{->>}[d]^p \\
		 	  T[1](I\times S^{n}) \ar[rr]_{h}\ar@{-->}[urr]^{\mathcal L^H(h,\tau)} && B[1].} 
		 	$$
		In particular, the restriction of $\mathcal L^H(h,\tau)$ to $ T[1](\{0\} \times S^n)$, which is $ \tau$, is homotopic to its restriction to   $ T[1](\{1\} \times S^n)$. The latter being mapped by $P$ to $ \underline{\ell}$ it takes values in  $ {\NQ{T}}_\ell$, so that $ [\tau]$ lies in the image of ${i}_* $.\\

		\noindent {\bf Exactness at $\pi_{n-1}(\NQ{T}_\ell,m)$.}
		Let $ \tau : T[1]S^{n-1}\to {\NQ{T}}_l$ be such that $i_* ( \tau ) =0 $.
		By Lemma \ref{lem:hotint} there exists an NQ-morphism $h : T[1](I \times S^{n-1})\to \NQ{M}$ whose restriction to $T[1](\{t\} \times S^{n-1}) $ is $\underline{m}$ for all $t $ near $0$, and whose restriction to $T[1](\{t\} \times S^{n-1}) $ is $ \tau$ for all $t$ near $1$. In particular, this implies that $ P_* \circ h:T[1](I\times S^{n-1})\to B[1]$ is constantly equal to $\ell$ on neighborhoods of $T[1](\{0\} \times S^{n-1}) $ and $T[1](\{1\} \times S^{n-1})$, defining therefore some NQ-manifold morphism $ \sigma \colon T[1]S^n \to B[1] $. The  NQ-morphisms $\delta|_m(\sigma)$ and $\tau$ are homotopic in $ \NQ{M}$ through a homotopy as in the second item of Lemma \ref{lem:hotint}, so that they are homotopic in $ {\NQ{T}}_\ell$.\\

		\noindent {\bf The relation $\delta(\pi_2(B[1], \ell))\subset \mathrm{Center}(\Pi(\NQ{T}_\ell))$.} Consider $[\sigma] \in \pi_2(B[1],\ell)$, and $\tau:T[1]I\to \NQ{T}_\ell$ a path from $m$ to $m'$. We have to show that $ \delta|_{m'}(\sigma)\circ \tau $ is homotopic to $\tau\circ \delta|_m(\sigma)$ in $\NQ{T}_\ell$. By Lemma \ref{lem:hotint}, it suffices to find a homotopy in $\NQ{M}$, whose image through $P$ is homotopic to a constant map in $B$. Such a homotopy is given by the concatenation of homotopies
				 	$$ \xymatrix{
				 	\delta|_m(\sigma)*\tau \ar[rrr]^{\mathcal L^H(\sigma,\underline{m})^{-1} * id_\tau}
				 	&&& \underline{m} * \tau \ar[r]^\sim
				 	& \tau*\underline{m'}\ar[rrr]^{id_\tau*\mathcal L^H(\sigma,\underline{m'})}
				 	&&& \tau *	\delta|_{m'}(\sigma) ,	} 
		 	$$
		 	where $*$ denotes both the concatenation of paths and homotopies. Recall that the concatenation  $b*a$ is a representative of the product $[a]\circ [b]$ in $\Pi_1$.
\end{proof}

\subsection{Examples of homotopy groups of NQ-manifolds}

In this subsection, we survey some examples and special cases of the above, mostly already present in the literature. 

	$$
	 \resizebox{0.99\textwidth}{!}{
	\xymatrix@C=1pt{ 
		&&\hbox{Lie $\infty$-algebroids} \ar[ld]^{\supset}\ar[rd]^{\supset}&&\\
		&\txt{Lie algebroids,\\ Ex. \ref{ex:LieAlgebroid}}\ar[rd]^{\supset}\ar[ld]^{\supset}&&\txt{Lie $\infty$-algebras, \\Ex \ref{ex:LieInfinityAlgebra}}\ar[ld]^{\supset}\ar[rd]^{\supset}&\\
		\txt{Tangent Lie algebroids,\\ Ex. \ref{ex:manifold}}&&\txt{Lie algebras,\\ Ex. \ref{ex:LieAlgebra}}&&\txt{Nilpotent Lie $\infty$-algebras,\\ Ex. \ref{ex:LieAlgebraNilpotent}}
	}
	}
	$$

\begin{example}
	\label{ex:manifold}
	\normalfont 
	Let $M$ be a manifold. Then the fundamental groupoid, fundamental groups, resp. $n$-th homotopy groups of $T[1]M$ are the fundamental groupoid, fundamental group resp. $n$-th homotopy groups of the manifold $M$.
\end{example}

\begin{example}
	\label{ex:LieAlgebra}
	\normalfont
	The fundamental groupoid and the fundamental group of a Lie algebra ${\mathfrak g}$ (seen as an NQ-manifold $\mathfrak g[1]$ over a point $pt$) is the simply-connected Lie group $\mathbf G$ integrating $ {\mathfrak g}$, see \cite{1707.00265,MR1973056}.
	For all $n \geq 2$, the $n$-th homotopy groups of ${\mathfrak g}[1]$ coincide with the $n$-th homotopy groups of the Lie group $\mathbf G$ (see \cite{Zhu}, Example 3.5).
	In equation:
	 $$ \mathbf \Pi (\mathfrak g[1]) = \pi_1(\mathfrak g[1], pt) = \mathbf G \hbox{ and for all $n \geq 2$: } \pi_n(\mathfrak g[1],pt)=\pi_n(\mathbf G, 1_\mathbf G). $$
\end{example}

\begin{example}
	\label{ex:LieAlgebroid}
	\normalfont
	For a Lie algebroid $A$ our definitions of fundamental groupoid, fundamental groups and $n$-th homotopy groups reduce to the definition given in \cite{1707.00265,MR1973056, Zhu}. 
	The fundamental groupoid is the source $s$-connected topological groupoid $\mathbf G$ that integrates the Lie algebroid \cite{1707.00265}: it is a smooth manifold in a neighborhood of $ M$. This idea has been widely used in the literature (cf. e.g. \cite{MR1973056, zbMATH01686785}).

	If the Lie algebroid is \emph{integrable}, that is when $ \mathbf G$ is a Lie groupoid, it is proven in \cite{Zhu} that the fundamental group at a given point $m$ is the isotropy $I_m( \mathbf G)= \mathbf G |_m^m$ of this groupoid at $m$, and the homotopy groups are the homotopy groups of $\mathbf G |_m = s^{-1}(m)$ of the fibers of the source map.
	In equation:
	 $$ \mathbf \Pi (A[1])  = \mathbf G, \pi_1(A[1], m)=I_m(\mathbf G) \hbox{ and for all $n \geq 2$: } \pi_n(A[1],m)=\pi_n(\mathbf G |_m, m) .$$
\end{example}

Let us state an immediate consequence of this example:
\begin{proposition}
\label{prop:Homot_grs_Of_Lie_oid}
Let $A$ be a Lie algebroid which is longitudinally integrable (\emph{i.e.}~the restriction of the fundamental groupoid $  \mathbf G$ of $A$ to every leaf is a smooth manifold).
Then for all $ n \geq 2$ and $m \in M$, we have $\pi_n(A[1],m)\simeq  \pi_n( \mathbf G|_m , m)  $ and  $\pi_1(A[1], m)=I_m(\mathbf G)$.  
\end{proposition}
\begin{proof}
It suffices to apply the second part of Example \ref{ex:LieAlgebroid} to the restriction of $A$ to the leaf $L$ through~$m$.
\end{proof}

\begin{example} Let us review some results by Berglund
	\label{ex:LieAlgebraNilpotent}.
	\normalfont
	\cite{Berlund}
	For a \emph{nilpotent Lie $\infty$-algebra} $ \mathfrak g_\bullet$, \emph{i.e.~}a  Lie $\infty$-algebra concentrated in degrees less or equal to $-2$:
	$$ \stackrel{ \mathrm  d}{\to} {\mathfrak g}_{-3}   \stackrel{  \mathrm  d}{\to} {\mathfrak g}_{-2} \to 0 $$
	the fundamental groupoid and fundamental groups are reduced to $0$ and, for all $n \geq 2$, the $n$-th homotopy groups are isomorphic to the cohomologies $ H^{-n}({\mathfrak g}_\bullet)$ of the above complex in degree $-n$.
	In equation:
	 $$ \mathbf \Pi (\mathfrak g_\bullet) = \pi_1(\mathfrak g_\bullet, pt) = \{ 0 \} \hbox{ and for all $n \geq 2$: } \pi_n(\mathfrak g_\bullet,m)= H^{-n}(\mathfrak g_\bullet) .$$
	
	Explicitly, an NQ-manifold morphism $ \sigma \colon S^n \to {\mathfrak g}_\bullet $ for $n \geq 2$ decomposes as $ \sum_{i=-n}^{-2} \alpha_i  $ with 
	 $\alpha_n \in  \Omega^n (S^n) \otimes {\mathfrak g}_{-n} , \dots, \alpha_2 \in  \Omega^{2} (S^n) \otimes {\mathfrak g}_{-2}   $.
	 Then $ \int_{S^n} \alpha_n  \in {\mathfrak g}_{-n} $ is a $d$-closed element and its class in $ H^{-n}({\mathfrak g}_\bullet)$  entirely determines the class up to homotopy of $\sigma $. 
\end{example}

\begin{example}
	\label{ex:LieInfinityAlgebra}
	\normalfont
	Consider an NQ-manifold $\mathfrak g_\bullet$ over a point $pt$, \emph{i.e.~}a negatively graded Lie $\infty$-algebra:
	$$  \dots  \stackrel{d}{\to} {\mathfrak g}_{-2}   \stackrel{d}{\to} {\mathfrak g}_{-1}.$$

Let us give a long exact sequence describing $\pi_n(\mathfrak g_\bullet,pt)$.
According to the homotopy transfer theorem (see, e.g. \cite{zbMATH06043075}), this Lie $\infty$-algebra is homotopy equivalent to a ``minimal model'', \emph{i.e.}~a Lie $\infty$-algebra with trivial differential constructed on the cohomology of the previous complex: 
$$ \dots \stackrel{0}{\to}  H^{-2}({\mathfrak g_\bullet}) \stackrel{0}{\to}   H^{-1}({\mathfrak g_\bullet}) .$$
The $2$-ary bracket of this Lie $\infty $-algebra does not depend on the choices of a minimal model and satisfies the graded Jacobi identity, so that $H^{-1}({\mathfrak g}_{\bullet}) $ is a Lie algebra.

The projection onto the minus one component is a  surjective submersion of a Lie $\infty$-algebra over a Lie algebra:    
	  \begin{equation} 
 \label{eq:Subm} \oplus_{i\geq 1}  H^{-i}({\mathfrak g}_{\bullet})  \to  H^{-1}({\mathfrak g}_{\bullet}) . 
  \end{equation} 
The fiber is the nilpotent Lie $ \infty$-algebra $ \oplus_{ i\geq 2}  H^{-i}({\mathfrak g}_{\bullet}) $. There is a unique Ehresmann connection (which is of course complete in view of Proposition \ref{prop:existeEhresmann}) given by $H= H^{-1}({\mathfrak g}_{\bullet})$ itself. We are therefore in the setting of  Theorem \ref{thm:snake} with all base manifolds being points, so that there is a long exact sequence of groups
	 $$ \dots \to H^{-n}( {\mathfrak g}_{\bullet}) \to   \pi_n( {\mathfrak g}_{\bullet},pt)  \to  \pi_{n}(\mathbf G, 1_{\mathbf G}) \stackrel{\delta}{\to} H^{-n+1}( {\mathfrak g}_{\bullet} ) \to \cdots $$
	 $$ \cdots  \to  \pi_2( {\mathfrak g}_{\bullet}, pt) \to 0 \to 0 \to \pi_1( {\mathfrak g}_{\bullet},pt) \to \mathbf G  \to 0 ,$$
	where $\mathbf G $ is the simply-connected Lie group integrating the Lie algebra $H^{-1}({\mathfrak g}_\bullet) $. Here, we used Example \ref{ex:LieAlgebraNilpotent}, Example \ref{ex:LieAlgebra} and the fact that $\pi_2(\mathbf G)=0$ for any Lie group.
	In particular, $ \pi_1( {\mathfrak g}_{\bullet},pt) \simeq \mathbf G  $.
\end{example}

\section{Higher holonomies of singular leaves}
\label{sec:leaf}

According to \cite{LLS}, a universal NQ-manifold can be associated to most singular foliations, in particular to those generated by real analytic local generators - which form a large class. 

We recall these constructions in Subsection \ref{sec:recapfol}, then define the homotopy groups of a singular foliation as the homotopy groups its universal NQ-manifold $\NQ{U}^\fol{F}$ in Subsection \ref{sec:homotopyGroups}.
In Subsection \ref{sec:isotropy}, we discuss how these  homotopy groups can be computed with the help of several long exact sequences derived from Theorem \ref{thm:snake}.
 
 We then consider a locally closed leaf $L$ of $ {\fol{F}}$. Upon replacing $M$ by a neighborhood of $L$, we can assume that there exists a surjective submersion $p \colon M \to L$ such that the composition of the anchor NQ-morphism $\NQ{U}^\fol{F}\to T[1]M$ with the surjection $T[1]p\colon T[1]M\to T[1]L$ is a fibration of NQ-manifolds over a Lie algebroid.
We first define and study Ehresmann connections for singular leaves in Section \ref{sec:Ehresmann_leaf}.
In Section \ref{sec:LESsing}, we show that Theorem \ref{thm:snake}, when applied in such a neighborhood of a singular leaf, gives a long exact sequence: we call higher holonomies of $\fol{F}$ at $L$ its connecting homomorphisms. 
We explain how this is related to the Androulidakis-Skandalis' holonomy groupoid, and how a Lie groupoid defining the singular foliation helps to construct it.

We will finish this section by focusing on several particular aspects of the higher holonomies, that reduce to well-known constructions, e.g.~the period map of Cranic-Fernandes \cite{MR1973056}.

\subsection{The universal NQ-manifold and homotopy groups of a singular foliation}

\subsubsection{The universal NQ-manifold  of a singular foliation}

\label{sec:recapfol}

 We recall and adapt the main results of \cite{LLS}.
Throughout this subsection,  ${\fol{F}}  $ is a singular foliation on a manifold $M$.

\begin{definition}\label{def:resolution} (Definition 2.1 in \cite{LLS})
	 A geometric resolution of ${\fol{F}}$ is a finite family of vector bundles $(E_{-i})_{i\geq 1}$ together with vector bundle morphisms $ E_{-i} \overset{d}{\to} E_{-i+1}$ and $\rho \colon E_{-1} \to TM$ such that the following sequence is a resolution of $ {\fol{F}}$ in the category of $\Cinfty{M} $-modules:
 $$ \dots \to \Gamma(E_{-3})  \to \Gamma(E_{-2}) \to  \Gamma(E_{-1}) \to  {\fol{F}}. $$
\end{definition}

\begin{remark}
\normalfont
 By the Serre-Swann theorem, a geometric resolution 
is a finite length projective resolution of ${\fol{F}}$ by finitely generated  $\Cinfty{M} $-modules.
\end{remark}

\begin{theorem}\label{theo:exists}
(Theorem 2.4 in \cite{LLS}) 
Every locally real analytic singular foliation on a manifold $M$ of dimension $n$ admits a geometric resolution of length $ n+1$ on every relatively compact open subset.
 \end{theorem}

Among the NQ-manifolds over $M$  {whose basic singular foliation is included} in $\fol{F}$, there is a distinguished class, which we now define.

\begin{definition}
 A \emph{universal NQ-manifold of a singular foliation  $ {\fol{F}}$} is an NQ-manifold $\NQ{U}^\fol{F}$ such that for any associated Lie $\infty$-algebroid $(E_\bullet,l_\bullet,\rho)$ the underlying sequence  $E_{-i}\overset{l_1}{\to} E_{-i+1}$ together with the anchor map $E_{-1}\overset{\rho}{\to} TM$ form a geometric resolution of $ {\fol{F}}$.
\end{definition}

The use of the word ``universal'' is justified by the following Theorem:

\begin{theorem}\label{theo:onlyOne}
(Theorem 1.7 in \cite{LLS})
	\texttt{	"Universality"}.
	Universal NQ-manifolds of $ {\fol{F}}$ are terminal objects in the category where objects are NQ-manifolds over $M$ inducing sub-foliations of $ {\fol{F}}$, and morphisms are homotopy classes of NQ-manifold morphisms over $M$.
\end{theorem}

Since in a category two universal objects are canonically isomorphic:

\begin{corollary} \label{coro:unique}
(Corollary 1.8 in \cite{LLS}) {\texttt{"Uniqueness"}}.
	Two  universal NQ-manifolds of a singular foliation $ \mathcal{F}$ are homotopy equivalent in a unique up to homotopy manner.
\end{corollary}

 We will denote a universal NQ-manifold of $\fol{F}$ by $\NQ{U}^{\fol{F}}$.
Singular foliations that admit a universal NQ-manifold admit a geometric resolution. 
In fact, the converse also holds:

\begin{theorem}  
(Theorem 1.6 in \cite{LLS}) \texttt{"Existence"}.
	\label{theo:existe}
If $\fol{F}$ admits a geometric resolution $(E_\bullet, d,\rho)$, then it admits a universal NQ-manifold $\NQ{U}^\fol{F}$ such that $\sh{\NQ{U}^\fol{F}}=\Gamma(\cdot, S^\bullet E^*_\bullet)$ and $l_1=d$ and $\rho_{\NQ{U}^\fol{F}}=\rho$.
\end{theorem}

We finish this section by two lemmas about restrictions of the universal NQ-manifold of a singular foliation to a transverse sub-manifold and to a leaf.

For any locally closed submanifold $ S \subset M$, we let $ {\fol{F}}|_S  \subset {\mathfrak X}(S)$ be the space obtained by restricting to $S$ vector fields in $ {\fol{F}}$ which are tangent to $S$.  We say that a submanifold $S \subset M$ is  \emph{${\fol{F}}$-transversal} if $ T_m S + T_m {\fol{F}} = T_m M $ for all $ m \in S$.
For a generic submanifold $S$, the space $\fol{F}|_S$ might not be finitely generated and hence might not be a singular foliation. However: 
\begin{lemma}\label{lem:transversefol0}
 The restriction  $ {\fol{F}}|_S$  to an ${\fol{F}}$-transversal $S$ 
is a singular foliation on $S$. Moreover, the restriction of any universal NQ-manifold of $ {\fol{F}}$ to $S$ is a universal NQ-manifold of $ {\fol{F}}|_S$. 
\end{lemma}
\begin{proof}
The first statement is proven in \cite{AS09}, Proposition 1.10,  Item (b). The second statement follows from the trivial observation that for any geometric resolution $ (E,\mathrm d, \rho)$ of $ {\fol{F}}$ and any  ${\fol{F}}$-transversal $S$ as above, the complex:
 $$   \dots \to E_{-3} |_{S} \stackrel{\mathrm d}{\longrightarrow} E_{-2} |_{S} \stackrel{\mathrm d}{\longrightarrow}  \rho^{-1}(TS)     \stackrel{\rho}{\longrightarrow}  TS $$
 is both a geometric resolution of  $ {\fol{F}}|_S$ and a sub-Lie $ \infty$-algebroid.
\end{proof}

Let $L$ be a leaf of $ \fol{F}$. We recall from \cite{AZ13} the following construction:
there is a unique transitive Lie algebroid over $L$, called \emph{holonomy Lie algebroid of $L$}, denoted by $ A_L$, whose sheaf of sections  is $ \fol{F} / I_L \fol{F}$, with $ I_L$ the sheaf of smooth functions on $M$ vanishing along $L$.

\begin{example}
\label{ex:regularLeafHolonomy}
\normalfont
    For a regular\footnote{A regular point is a point around which all leaves have the same dimension (and are called regular leaves). Recall that if a geometric resolution of finite length exists, then all regular leaves have the same dimension.} leaf $L$, the holonomy Lie algebroid is the tangent Lie algebroid. In equation: $ A_L =TL$. 
\end{example}

Consider the restriction $\mathfrak i_L \NQ{U}^\fol{F} $ of any universal NQ-manifold $\NQ{U}^\fol{F} $.
This is an NQ-manifold over $L$, which is transitive over $L$. 

\begin{lemma}\label{lem:transversefol}
The restriction $\mathfrak i_L \NQ{U}^\fol{F}$ of any universal NQ-manifold $\NQ{U}^\fol{F} $  to a leaf $L$ of a singular foliation is an NQ-manifold such that for any $E_\bullet$ satisfying $\sh{\NQ{U^\fol{F}}}=\Gamma(\cdot, S^\bullet E_\bullet^*)$:
 \begin{enumerate}
     \item the differential $ \mathrm d: \mathfrak i_L E_{-i-1} \to \mathfrak i_L E_{-i}$ has constant rank along $L$,
     \item the quotient space $ A_L := \mathfrak i_L E_{-1} / \mathrm d \mathfrak i_L E_{-2}$:
     \begin{enumerate}
         \item comes equipped with a transitive Lie algebroid structure,
         \item which does not depend on the choice of a universal NQ-manifold $\NQ{U}^\fol{F} $,
         \item  and coincides with Androulidakis-Zambon's holonomy Lie algebroid $A_L$ of the leaf $L$.
     \end{enumerate}
 \end{enumerate}
\end{lemma}
\begin{proof}
Let $ \tau : TL \to E_{-1}$ be a section of $ \rho : E_{-1} \to TL $. Then the bilinear map
 $$ \nabla_u^{(i)} e := [\tau ( u ) , e ] \hbox{ for all $u \in \mathfrak X (L) $ and $e \in \Gamma(E_{-i}) $ }$$
defines a connection on $ E_{-i}$ for all $i \geq 1$. An easy computation gives:
 $$ \nabla_u^{(i+1)} (\mathrm d e) =  [\tau ( u ) , \mathrm d e ] = \mathrm d [\tau ( u ) , e ] = \mathrm d \nabla^{(i)}_u e .$$
 Since the covariant derivative of $\mathrm  d $ is $0$, its rank has to be constant. This proves the first item. 
For the second item, see Proposition 4.18 in \cite{LLS}.
\end{proof}

\subsubsection{Definition of the homotopy groups of a singular foliation}
\label{sec:homotopyGroups}

Throughout this section, $ {\fol{F}}$ is a singular foliation  that admits a geometric resolution: by Theorem \ref{theo:existe}, it also admits a universal NQ-manifold $\NQ{U}^{\fol{F}} $. Since homotopy equivalent NQ-manifolds have isomorphic homotopy groups and fundamental groupoid by Corollary \ref{cor:homotopyGroups},
and since any two universal NQ-manifolds are homotopy equivalent by Corollary \ref{coro:unique}, 
the groups $\pi_n(\NQ{U}^\fol{F}, m) $, $\Gamma(\pi_n(\NQ{U}^\fol{F}) ) $ and the  fundamental groupoid $\mathbf\Pi (\NQ{U}^\fol{F}) $ do not depend on the choice of a universal NQ-manifold. In particular, this justifies the following definition and the following notations:

\begin{definition}
For every singular foliation $\fol{F} $ that admits a geometric resolution, we define 
 \begin{enumerate}
     \item the $n$-th homotopy group $ \pi_n ({\fol{F}}, m) $ at $m$, 
     \item its space of smooth sections $\Gamma(\pi_n(\fol{F})) $,
     \item  and the fundamental groupoid $\Pi(F)$ of $\fol{F} $,
 \end{enumerate}  
 to be those of any one of its universal NQ-manifolds.
\end{definition}

In \cite{AS09}, Androulidakis and Skandalis constructed the holonomy Lie groupoid of a singular foliation: It is a topological groupoid whose leaves are the leaves of $ {\fol{F}}$, and whose restriction to a given leaf $L$ is a smooth manifold in a neighborhood of the identity. 
According to Proposition 4.37 in \cite{LLS}, the fundamental groupoid $\Pi(\fol{F}) $ of $\fol{F} $ is the universal cover of the holonomy groupoid of Androulidakis and Skandalis\footnote{By universal cover, we mean that for each leaf $m \in M$, $\Pi(\fol{F}) |_m=s^{-1}(m) $ is the universal cover of the fiber over $m$ of  Androulidakis and Skandalis' holonomy groupoid}. In view of this identification, and since the holonomy Lie groupoid is a well-established mathematical object, we prefer to adopt the terminology as follows:

\begin{convention}
\label{con:newName}
From now on, we will call $\mathbf F=\Pi(\fol{F})  $ the \emph{universal holonomy groupoid of $\fol{F} $} rather than ``the fundamental groupoid of $\fol{F} $''.
\end{convention}

\begin{remark}
\normalfont
\label{rmk:holonomyLieAlgerbroid}
The universal holonomy groupoid $\mathbf F$ can be characterized as follows: it induces the same foliation as $ \fol{F}$ and its  restriction to a leaf $L$ is the fundamental groupoid of the holonomy Lie algebroid $A_L$.
\end{remark}

Lemma \ref{lem:restrLeaf} implies that $ \pi_n(\fol{F},m)$ only depends on the restriction of the universal NQ-manifold $ \NQ{U}^\fol{F}$ to the leaf $L$ to which $m $ belongs:

\begin{lemma}
For every universal NQ-manifold $ \NQ{U}^\fol{F}$ of a singular foliation $ \fol{F}$, and every point $m \in M$, we have:
 $$  \pi_n(\fol{F}, m) = \pi_n ( \mathfrak i_L \NQ{U}^\fol{F}, m)$$
 where $L$ is the leaf through $ m$.
\end{lemma}

\begin{example}
\normalfont
\label{ex:regular}
	When $ \fol{F}$ is a regular foliation, we have $ \fol{F} = \Gamma(A)$ with $A\subset TM$ an integrable distribution. In this case, $A[1]$ is itself a universal NQ-manifold of $ \fol{F}$. The universal holonomy groupoid is then the universal cover of the holonomy groupoid of $\fol{F} $ as defined in \cite{zbMATH01956612}. Also, at a point $m \in M$, the homotopy groups of $ \fol{F}$ are the homotopy groups of the leaf through $m$.
\end{example}

\begin{example}
\normalfont
\label{ex:regularleaf}
By the previous example, we see that 
 $ \pi_n( \fol{F},m) = \pi_n(L,m) $
 for every regular point $m \in M$, with $L$ being the leaf through $m$.
\end{example}

\begin{example}
\normalfont
\label{ex:Debord1}
We say that a singular foliation is \emph{Debord} or \emph{projective} when it is projective as a module over functions. By the Serre-Swann theorem, it means that $ {\fol{F}}$ comes from a Lie algebroid $A$ whose anchor is injective on a dense open subset. By construction, the Lie algebroid $A[1]$ is a universal NQ-manifold of $ {\fol{F}}$. Moreover, Theorem 1 in \cite{D} says that $A$ integrates into a smooth (source simply-connected) Lie groupoid, which has to coincide with $ \mathbf F$. In view of Example \ref{ex:LieAlgebroid}, we have\footnote{As usual for Lie groupoids, $ \mathbf F |_m = s^{-1}(m) \subset \mathbf F$ is the source fiber of $m$, and  $\widetilde{I_m({\mathbf F})}$ is the universal cover of the isotropy of $I_m(\mathbf F) $ of $ \mathbf F$ at $m$.} $ \pi_n(\fol{F},m) = \pi_n( {\mathbf F}|_m ,m)$ for all $ n \geq 2$ and $ \pi_1(\fol{F},m)$ is the isotropy ${I_m({\mathbf F})}$ of $\mathbf F$ at $m$.
\end{example}

\subsubsection{Computation of the homotopy groups of a singular foliation}
\label{sec:isotropy}

The purpose of this subsection is to describe several exact sequences that lead to the computation of the homotopy groups $ \pi_n({\fol{F}},m)$, when the following two ingredients are given:
\begin{enumerate}
\item the isotropy Lie $\infty$-algebra of $ \fol{F}$ at $m \in M$, 
 \item the universal holonomy groupoid $ \mathbf F$ of $\fol{F} $, see Convention \ref{con:newName}, more precisely its restriction to the leaf $L$ through $m$. \end{enumerate}

Let us briefly recall how the first of these objects is constructed. The \emph{isotropy Lie $\infty$-algebra of $\fol{F} $ at $m$} is introduced in \cite{LLS}, Section 4.2, by the following spaces and brackets:
    \begin{enumerate}
        \item  The cohomology groups $ H^\bullet(\fol{F}, m) := \oplus_{ n \geq 1}  H^{-n}( {\fol{F}},m )$ of any geometric resolution $(E_\bullet, \mathrm d, \rho) $ of $\fol{F}$, evaluated at $ m \in M$:
    \begin{align}
        \label{eq:Hn} 
        &H^{-n}(\fol{F},m)= H^{-n}(E_\bullet|_m,d|_m)=\frac{\mathrm{ker}( d|_m: E_{-n}|_m\to E_{-(n-1)}|_m)}{\mathrm{im}(d|_m: E_{-(n+1)|_m}\to  E_{-n}|_m)} ~~~~ \mathrm{~for~}n\geq 2\\
        &H^{-1}(\fol{F},m)=\nonumber H^{-1}(E_\bullet|_m,d|_m)=\frac{\mathrm{ker}( \rho|_m: E_{-1}|_m\to T_mM)}{\mathrm{im}(d|_m: E_{-2|_m}\to  E_{-1}|_m)} 
    \end{align}
     coincides by construction with $ Tor^{\bullet}({\fol{F}}, {\mathbb R})$, with ${\mathbb R}$ being made a $\Cinfty{M}$-module through evaluation at $m$ (see Remark 4.9 in \cite{LLS}).
     In particular, the graded vector space $H^\bullet(\fol{F}, m)$ does not depend on the chosen resolution and depends on the singular foliation $ {\fol{F}}$ only, which explains the notation. 
        \item The Lie $ \infty$-algebroid brackets of any universal Lie $\infty$-algebroid of $ \fol{F}$ restrict to the point $m$ to yield a Lie $\infty $-algebra, called \emph{isotropy Lie $\infty$-algebra}. Its \emph{minimal model} is by construction a Lie $\infty$-algebra structure on $H^\bullet(\fol{F}, m)$ whose un-ary bracket is zero. In particular, its $2$-ary bracket satisfies the graded Jacobi identity. Any two minimal models are strictly isomorphic through an isomorphism whose linear map is the identity map of  $ H^{\bullet}( {\fol{F}},m )$: in particular, the $2$-ary graded Lie bracket is canonically attached to $ {\fol{F}}$.
    \end{enumerate}

We now give a long exact sequence that computes the homotopy groups of a singular foliation.

\begin{proposition}
\label{prop:computingISotropies}
Let ${\fol{F}}  $ be a singular foliation that admits a geometric resolution and $\mathbf F $ be the universal holonomy Lie groupoid. Then  {for any point $m$ in a leaf $L$,} there is a natural exact sequence of groups as follows:
\begin{align}
\label{eq:computingISotropies}
    \dots \to H^{-n}( {\fol{F}}, m) \longrightarrow   \pi_n( \fol{F},m) \stackrel{q}{ \longrightarrow } \pi_{n}({\mathbf F}|_m,m)  \longrightarrow H^{-n+1}( {\fol{F}},m) \longrightarrow \cdots \\
	  \cdots  \longrightarrow \pi_2( \fol{F}, m)    \stackrel{q}{ \longrightarrow } \pi_2(\mathbf F|_m,m)  \longrightarrow  0  \longrightarrow \pi_1( \fol{F}, m)  \longrightarrow {\mathbf{F}|_m^m}  \longrightarrow 0 ,\nonumber
\end{align}
	where  $ \mathbf F|_m = s^{-1}(m)$ is the $s$-fiber over $m$ of the source map, $\mathbf{F}|_m^m$ is the isotropy at $m$ of the universal holonomy Lie groupoid.
\end{proposition}
\begin{proof} Consider the restriction $\mathfrak i_L  \NQ{U}^\fol{F}$ of the universal NQ-manifold $ \NQ{U}^\fol{F} $ to the leaf $L$. 
The exact sequence \eqref{eq:computingISotropies} is a particular case of the exact sequence obtained in Theorem \ref{thm:snake} to the fibration of NQ-manifolds:
 $$ Q \colon \mathfrak i_L \NQ{U}^{\fol{F}} \longrightarrow A_L[1] = \left(\tfrac{{\mathfrak i}_L E_{-1}}{ \mathrm d \mathfrak i_L E_{-2}} \right)[1]  $$
 over the holonomy Lie algebroid of $L$. Its fiber over $ \ell$ is the nilpotent part of isotropy Lie $\infty$-algebra of $ \fol{F}$: for $ n \geq 2 $ its $\pi_n $ coincide with $H^{-n}(\fol{F} , m )$ in view of Example \ref{ex:LieAlgebraNilpotent}, while its $ \pi_1$ is trivial.
 Since the holonomy Lie algebroid is the Lie algebroid of the universal holonomy groupoid (see Remark \ref{rmk:holonomyLieAlgerbroid}), Example \ref{ex:LieAlgebroid} gives an expression of its homotopy groups.

 By Proposition 2.2. in \cite{MR3119886}, the holonomy groupoid of a singular foliation is longitudinally smooth. Since $\mathbf F $ is its universal cover by Proposition 4.37 in \cite{LLS}, $ \mathbf F|_m$ is a smooth manifold for all $m\in M$. Proposition \ref{prop:Homot_grs_Of_Lie_oid} gives therefore for all $ n \geq 2$ and $ m \in L$ an isomorphism $ \pi_{n}({\mathbf F}|_m,m)\simeq \pi_n(A_{L}, m) $ and 
 $ \pi_{1}({\mathbf F}|_m,m)\simeq I_m(\mathbf F) = \mathbf F|_m^m $.
 {One can now apply Theorem \ref{thm:snake}.}
\end{proof}
\begin{remark}\normalfont
\label{rmk:holonomyGpdPi2}
Let $L$ be a leaf of a singular foliation $ \fol{F}$, and $A_L$ be the holonomy Lie algebroid. Applying Theorem \ref{thm:snake} or Theorem 1.4 in \cite{Zhu} to the surjective submersion  of Lie algebroids:
 $$  A_L \to  T L $$
 whose fiber over $m$ is the isotropy Lie algebra $ H^{-1}(\fol{F}, m ) $,
we obtain the following long exact sequence:
\begin{align}
\label{eq:computingISotropies2}
    \dots \to  \pi_n(  \mathbf H_m , 1 ) \longrightarrow \pi_n( \mathbf F |_m, m )   \longrightarrow \pi_n( L , m  )  \longrightarrow \cdots \\
	 \cdots \longrightarrow  0 \longrightarrow \pi_2( \mathbf F|_m, m)   \longrightarrow  \pi_2( L,m)  \longrightarrow  \mathbf H_m \longrightarrow \mathbf F_m  \stackrel{t}{\longrightarrow}  L ,\nonumber
\end{align}
with $ \mathbf F|_m  = s^{-1}(m) \subset  \mathbf F_m$
and $\mathbf H_m$ the simply connected Lie group integrating the isotropy Lie algebra.
The second line of \eqref{eq:computingISotropies2} appears in Proposition 3.9. in \cite{AZ13}. Let us explain briefly the relation.  Equation \eqref{eq:computingISotropies2} yields an exact sequence:
\begin{align}
\label{eq:computingISotropies3}  0  \longrightarrow  \frac{ \pi_2( L,m) }{ \pi_2( \mathbf F|_m, m) }  \longrightarrow  \mathbf H_m \longrightarrow \mathbf F|_m^m  \longrightarrow  0  
\end{align}
which matches (3.6) in \cite{AZ13}: The left-hand term of \eqref{eq:computingISotropies3} is therefore what Androulidakis and Zambon call the essential isotropy group of $L$.
\end{remark}

\subsection{Ehresmann ${\fol{F}}$-connections}

\label{sec:Ehresmann_leaf}

\subsubsection{Definition and existence}

Let $L$ be a locally closed leaf of a singular foliation $ {\fol{F}}$.

\begin{definition}
An Ehresmann ${\fol{F}}$-connection near $L$ is a triple made of:
\begin{enumerate}
    \item[a)] A neighborhood $M_L$ of $L$ in $M$ equipped with a projection $p\colon M_L\to L$,
    \item[b)] An Ehresmann connection $H$ for $ p$ whose sections are elements of $ {\fol{F}}$.
\end{enumerate}
It is said to be a complete Ehresmann  ${\fol{F}}$-connection near $L$ if $H$ is complete.
\end{definition}

Let us discuss the existence of such objects. The first statement of the following Proposition was established and proved while discussing with Iakovos Androulidakis \cite{AndroulidakisPrivate}.

\begin{proposition}
\label{prop:existeEhresmann}
Any locally closed leaf $L$ of singular foliation $ {\fol{F}}$ admits an Ehresmann ${\fol{F}}$-connection in a neighborhood of $L$. It can be chosen to be complete, possibly in a sub-neighborhood $M_L $, if there exists a function $ \kappa \colon M_L \to {\mathbb R}$ vanishing along $L$
such that the restriction of $p \colon \{ \kappa \leq 1 \} \to L $ is a proper map satisfying either of the following assumptions:
\begin{enumerate}
    \item $ \kappa $ is constant along the leaves,
    \item for every $m \in\kappa^{-1}(1) $, there exists a vector field $ u \in {\fol{F}}  $ tangent to the fibers of $p$ such that $ u[  \kappa]|_m \neq 0 $.
\end{enumerate}
\end{proposition}
\begin{proof}
According to the tubular neighborhood theorem, there exists a neighborhood $U $ of the leaf $L$ in $M$ that comes equipped with a projection $p \colon U \to L$ which is a surjective submersion.
Now, there is a neighborhood $M_L$ of $L$ in $ U$ such that for all $m \in M_L$, $ T_m {\fol{F}} + {\mathrm{Ker}} (T_m p ) =T_m M $. 

{
For every $l\in L$ there exists a neighborhood $V_l$ and a local trivialization $u_1,...,u_n$ of $TL$. By definition of $M_L$, there exist $X_1,...,X_n\in \fol{F}$ $p$-related to $u_1,...,u_n$,
} defining therefore an Ehresmann $ {\fol{F}}$-connection $ H_m$ on $p^{-1}(V_l)$. 

Let $ (\phi_i,V_i)_{i \in I} $ be a partition of unity of $ M_L$, such that an Ehresmann $ {\fol{F}}$-connection $H_i$ of the previous type exists on $V_i$.
A global Ehresmann $ {\fol{F}}$-connection on $M_L$ is given as follows: it is the unique distribution for which the horizontal lift of $X \in {\mathfrak X}(L)$ is $ H(X) := \sum_{i \in I} \phi_i H_i ({X})$ with  $H_i({X})$ the horizontal lift of $X|_{p(V_i)} $ with respect to $H_i$ on $V_i$.  

If a function $\kappa  $ satisfying the first condition exists, then the horizontal lift $ \tilde{X}$ of every vector field $X \in {\mathfrak X}(L) $ satisfies $ \tilde{X}[\kappa]=0$, so that its flow preserves $ \{\kappa \leq 1\}$. Since $p: \{\kappa \leq 1\} \to L$ is a proper map, this proves that the flow of $X$ is defined if and only if the flow of $ \tilde{X}$ is defined.
 
Assume now that a function $\kappa  $ satisfying the second condition exists. This implies that $ \{\kappa =1\}$ is a submanifold of $ M_L$. Using partitions of unity, we can construct a vertical vector field $u \in {\fol{F}}$ such that $ u[\kappa]=1$ at any point of $  \{\kappa =1\}$. Now, let $H $ be an Ehresmann $ {\fol{F}}$-connection on $p \colon M_L \to L$. By assigning to a vector field $X \in \mathfrak X (L) $ the vector field $H(X) - H(X)[\kappa] \,   u $, we define a second Ehresmann ${\fol{F}} $-connection $ H_\kappa$ whose sections are tangent to the submanifold $\{ \kappa = 1 \} $. {The flows of these sections, preserve} the subset  $\{ \kappa \leq 1 \}  \subset M_L$. Since $p$ is a proper map on that subset, the flow of $X$ is defined if and only if the flow of  $ H_\kappa (X) $ is defined, implying the claim. 
 \end{proof}

\begin{example}
\normalfont
\label{ex:Tsphere}
Let $M=TS^n$ be the tangent space of the sphere with its standard metric, and let $ \kappa : TS^n \to {\mathbb R}$ be the square of the norm.
Let $ {\fol{F}}$ be the singular foliation on $M=TS^n$ of all vector fields $X$ that satisfy $ X[\kappa]=0$. The leaves of this singular foliation are the sets $ \kappa^{-1}(\lambda)$ for $ \lambda \in {\mathbb R}_+$. Moreover,  $\kappa^{-1}(0)$, \emph{i.e}.~the zero section, is the unique singular leaf. The function $ \kappa$ obviously satisfies the first condition in Proposition \ref{prop:existeEhresmann}, so that complete Ehresmann ${\fol{F}}$-connections exist. Indeed, the horizontal distribution associated to the Levi-Civita connection is a complete Ehresmann $ {\fol{F}}$-connection.
\end{example}

\begin{example}
\normalfont
Let $\fol{F}$ be the singular foliation on $M=TS^n$ of all vector fields tangent to the zero section $S^n\subset TS^n$. The zero section is a leaf, and the function $\kappa$ defined in Example \ref{ex:Tsphere} now satisfies the second condition of Proposition \ref{prop:existeEhresmann}. Hence a complete Ehresmann $ {\fol{F}}$-connection exists. Indeed, the horizontal distribution associated to any linear connection is a complete Ehresmann $ {\fol{F}}$-connection.
\end{example}

Consider a complete Ehresmann ${\fol{F}}$-connection $(M_L,p,H)$ near a locally closed leaf $L$. The fiber of $p $ over $\ell $ is an $ {\fol{F}}$-transversal to $ {\fol{F}}$ at $\ell$. The restricted singular foliation $ {\fol{F}}|_{p^{-1} \ell}$ shall be denoted $ {\fol{T}}_\ell$.  \\

 []{The following proposition extends the classical normal form theorem of singular foliations in the presence of a complete Ehresmann connection (\cite{Dazord, Cerveau, AS09}).}

\begin{proposition}\label{prop:loctriv}
Consider a complete Ehresmann ${\fol{F}}$-connection $(M_L,p,H)$ near a locally closed leaf $L$. Every $ \ell \in L$ admits a neighborhood $U \subset L$ and a local diffeomorphism $ p^{-1}(U) \simeq U \times p^{-1}(L) $ intertwining $ {\fol{F}}$ with the direct product foliation $ {\mathfrak X}(U) \times  {\fol{T}}_\ell$.
\end{proposition}
\begin{proof}
It  suffices to prove the following result: ``For $L \simeq J^n$  with $ J = ]-1,1[$, if a complete Ehresmann ${\fol{F}}$-connection $(M_L,p,H)$ over $L$ exists, then  a flat\footnote{A flat Ehresmann connection is a horizontal distribution whose sections are closed under the Lie bracket of vector fields.} complete Ehresmann ${\fol{F}}$-connection exists''.

We prove this statement by induction on $n$. For $n=1$, since every  Ehresmann connection is flat, the result is straightforward. 

Assume that it holds true for some $n \in {\mathbb N}$.
Consider a complete Ehresmann ${\fol{F}}$-connection $(M_L,p,H)$ over $L= J^{n+1}  $.
Let $t \in J$ stand for the variable in the last copy of $J$. The horizontal lift (with respect to $H$) of the vector field $\tfrac{\partial}{\partial t}$ being complete by assumption, its flow defines a family of diffeomorphisms $ \Psi_t \colon p^{-1}(J^n \times \{ s \} ) \to p^{-1}(J^n \times \{s+t\})$ for all $t,s \in  J$ such that $ s+t \in J$. Being the flow of an element in $ {\fol{F}}$, $\Psi_t $ preserves $ {\fol{F}}$, so that:
 \begin{equation}
     \label{eq:pro:Phi}
    (\Psi_t )_* \,  \colon \,  {\fol{F}}|_{p^{-1}(J^n\times  \{0\})} \simeq  {\fol{F}}|_{p^{-1}(J^{n}\times  \{t\})}.
 \end{equation}
By recursion hypothesis, there exists a flat complete Ehresmann ${\fol{F}}$-connection $(M_L,p,H')$ on the restriction 
$$ p^{-1}( J^n \times \{0\}) \to   J^n \times \{0\}   .$$ 
We use $ \Psi_t$ to transport this flat complete Ehresmann connection $\tilde{H}' $ over
 $$ p^{-1}( J^n \times \{t\}) \to   J^n \times \{t\} . $$
According to Equation \eqref{eq:pro:Phi}, $\tilde{H}' $ is a flat Ehresmann  ${\fol{F}}$-connection for all $t \in J$. Consider the distribution
 \begin{equation}\label{eq:final_dist} H'':= \langle H(\tfrac{\partial}{\partial t})\rangle \oplus  \tilde{H}' .\end{equation}
The distribution $ H''$
 \begin{enumerate}
     \item  is of rank  $ n+1$ and is horizontal with respect to $ p$,
     \item  is flat because $\tilde{H}' $ is flat over $ p^{-1}(J^n \times \{t\})$ for all $t \in J$ and is by construction invariant under the flow of $ H(\tfrac{\partial}{\partial t})$,
     \item has sections in $ {\fol{F}}$ because 
     $H(\tfrac{\partial}{\partial t}) \in {\fol{F}}$, and because sections of $\tilde{H}' $ are in ${\fol{F}}$,
     \item  is complete because $H(\tfrac{\partial}{\partial t})$ is a complete vector field and  $ \tilde{H}'$   is a complete flat Ehresmann connection on $ p^{-1}(J^n \times \{t\})$ for all $t \in J$. Since complete vector fields for a flat connection are a $ \Cinfty{L}$-module, this is enough to guarantee the completeness of $H'' $.
 \end{enumerate}
 This completes the proof.
\end{proof}

\begin{remark}
\label{rmk:Pedro}
\normalfont
When $\mathcal  F$ is a Debord foliation, it comes from a Lie algebroid with an anchor map injective on an open subset, then Theorem C in \cite{MR3897481} gives an alternative proof of Proposition \ref{prop:loctriv}. 
\end{remark}

\subsubsection{Ehresmann $ {\fol{F}}$-connections and NQ-manifolds}

Most singular foliations that appear in a given concrete problem are locally real analytic. Since locally real analytic singular foliations do admit a geometric resolution over a locally compact open subset by Theorem 2.4 in \cite{LLS}, the following theorem justifies the assumption (that we always make) about the existence of a universal NQ-manifold for $ \fol{F}$ on  $M_L$.

\begin{theorem}\label{thm:onetransverseisenough}
Consider a complete Ehresmann ${\fol{F}}$-connection $(M_L,p,H)$ near a compact leaf $L$.  The following statements are equivalent:
\begin{enumerate}
    \item[(i)] There exists an $ \ell \in L$ such that   $ {\fol{T}}_\ell$ admits a geometric resolution.
    \item[(ii)] The restriction of ${\fol{F}}$ to $M_L$ admits a universal NQ-manifold.
\end{enumerate}
In particular, if $ {\fol{T}}_\ell$ is a locally real analytic singular foliation, and $ p^{-1}(\ell)$ is relatively compact in $M$, then (ii) holds.
\end{theorem}
Of course, $ {\fol{T}}_\ell$ is a locally real analytic singular foliation
if  ${\fol{F}}$ is locally real analytic singular foliation and $p$ is, around each point, real analytic in some local coordinates (that do not need to patch in a real analytic manner).
\begin{proof}
It is obvious that (ii) implies (i). Conversely, Proposition \ref{prop:loctriv} implies that for any $\ell, \ell'\in L$, $\fol{T}_\ell$ and $\mathcal{ T}_{\ell'}$ are isomorphic and that any $\ell'\in L$ has a neighborhood $U_{\ell'}$  such that $p^{-1}(U_{\ell'})\subset M_L$ admits a geometric resolution. By compactness of $L$, $M_L$ can be covered by finitely many such $p^{-1}(U_{\ell'})$. The statement now follows from Theorem 2.4 in \cite{LLS}.
\end{proof}

\begin{proposition}
    \label{prop:compatEhresmann}
Consider an Ehresmann ${\fol{F}}$-connection $(M_L,p,H)$ near a locally closed leaf $L$. {Assume that $\NQ{U}$ is an NQ-manifold whose basic foliation is $\fol{F}$}
\begin{enumerate}
    \item The composition of the projections $\NQ{U} \to T[1]M_L$ and $T[1]p:T[1]M_L\to T[1]L$ is an NQ-manifold fibration $ P \colon \NQ{U} \to TL$ over the tangent Lie algebroid $TL$.
    \item There exists an Ehresmann connection $(E_\bullet, H_{\NQ{U}})$
    for $P \colon {\NQ{U}} \to T[1]L$ such that $ \rho_{\NQ{U}} ( H_{\NQ{U}} )  = H $.
    \item Moreover, any such Ehresmann connection $ H_{\NQ{U}}$ is complete for $P $ if and only if $(M_L,p,H)$ is a complete Ehresmann $ {\fol{F}}$-connection.
 \end{enumerate}
\end{proposition}
\begin{proof}
Like any NQ-manifold over $M_L$, $\NQ{U}$ has a natural NQ-manifold morphism to $T[1]M_L$. When picking a Lie $\infty$-algebroid $(E_\bullet, \l_k, \rho_{\NQ{U}}$ associated to $\NQ{U}$, this morphism  corresponds to the anchor map $\rho_{\NQ{U}}:{E_{1}} \to TM_L$. The projection $T[1]p:  \to T[1]L$ is a Lie algebroid morphism. Thus their composition is a Lie $\infty$-algebroid morphism. 
Since  $ p \colon M_L \to L$ is a surjective submersion and since $Tp \colon H \to TL $ is bijective and $ H_m \subset \rho_{{\NQ{U}}}(E_{-1}|_m) $, this Lie $\infty $-algebroid morphism is a surjective submersion. This proves the first item.

To prove the second item, it suffices to construct a vector bundle morphism $ \sigma: H \to E_{-1}  $ such that $ \rho \circ \sigma = {\mathrm {id}}_{H} $: Its image $ H_{\NQ{U}} = \sigma(H) \subset E_{-1}$ satisfies then the required conditions. 
Let us first construct such a section locally. Let $d$ be the dimension of $L$. Since $H$ is an Ehresmann $ {\fol{F}}$-connection, there exists for every $m \in M_L $ vector fields $X_1, \dots,X_d \in {\fol{F}}$  that generate $H$ on every point in a neighborhood $ U_m \subset M_L$. Let $ e_1, \dots, e_d \in \Gamma_{U_m}(E_{-1})$ be sections such that $ \rho(e_i) = X_i$ for all $i=1, \dots, d$. The vector bundle morphism:
 $\sigma_{U} \colon H \to E_{-1}$ mapping $X_i|_{m'} $ to $ e_i|_{m'}$ for all $m' \in U$ satisfies $\rho \circ \sigma_U ={\mathrm{id}}_H $ on $U$ by construction.
 Let $ (U_i, \phi_i)_{i \in I}$ be a partition of unity of $M$ such that each open subset $ U_i$ comes equipped with a section $\sigma_{i} \colon H \to E_{-1}$ as above. The vector bundle morphism:
  $$ \sigma := \sum_{i \in I} \phi_i \sigma_i $$ 
  satisfies $ \rho_{\NQ{U}} \circ \sigma = \sum_{i \in I} \phi_i \, \rho_{\NQ{U}} \circ \sigma_i =  \sum_{i \in I} \phi_i \, {\mathrm{id}}_H = {\mathrm{id}}_H$. This proves the second item.

The third item is an immediate consequence of Proposition \ref{prop:aboutComplete}
\end{proof}

\subsection{The higher holonomies of a singular leaf} 
\label{sec:LESsing}

Throughout this section, $L$ is a locally closed leaf in is a singular foliation $ {\fol{F}}$ on a manifold~$M$.

\subsubsection{Main theorem and definition}

The main result of this section is the following theorem.

\begin{theorem}
\label{theo:holonomies}
 Let $ {\fol{F}}$ be a singular foliation that admits a geometric resolution. Let $L$ be a locally closed leaf that admits a complete Ehresmann ${\fol{F}}$-connection $(M_L,p,H)$.
For every $\ell \in L$,
    there exist group morphisms 
    \begin{equation}
\label{eq:def_Holonomy}
    Hol : \pi_n (L ,\ell )  \to \Gamma \left( \pi_{n-1}(  {\fol{T}}_{\ell}) \right) \end{equation}
    such that for all $m\in p^{-1}(\ell)$ the following sequence is exact
    \begin{align}\label{seq:leaf}
      \dots \stackrel{Hol|_{m}}{\longrightarrow}   \pi_n( {{\fol{T}}}_\ell,m )  \stackrel{i}{\to} \pi_n({\fol{F}}|_{M_L}  , m)  \stackrel{P}{\to} \pi_n(L,\ell)  \stackrel{Hol|_{m}}{\longrightarrow}  \pi_{n-1}( {{\fol{T}}}_\ell,m ) \to \dots 
    \end{align}
Here $ {\fol{T}}_{\ell}$ is the transverse singular foliation on $p^{-1}(\ell)$.
\end{theorem}
\begin{proof}
Let $ \NQ{U}^\fol{F}$ be a universal NQ-manifold of $\fol{F} $.
According to the first item of Proposition \ref{prop:compatEhresmann}, the morphism of NQ-manifolds $ \NQ{U}^\fol{F} \to T[1]L$.
is a fibration of NQ-manifolds. According to the second item in Proposition \ref{prop:compatEhresmann},  $P$ admits an Ehresmann connection. According to its third item, the latter 
Ehresmann connection is complete. 

We can therefore apply Theorem \ref{thm:snake}.
Since the base Lie algebroid is the tangent Lie algebroid $TL \to L$, its homotopy groups and the usual homotopy groups of the manifold $L$, see Example \ref{ex:manifold}. Also, the kernel of $ P$ over a given point $ \ell \in L$ is the universal NQ-manifold of the transverse singular foliation $ {\fol{T}}_\ell$  on $p^{-1}(\ell) $ by Lemma \ref{lem:transversefol}. Its homotopy groups are therefore the homotopy groups of the transverse foliation $ {\fol{T}}_\ell$ on $ p^{-1}(\{l\})$. Theorem \ref{thm:snake} yields a family of group morphisms as in Equation \eqref{eq:def_Holonomy}  that make the sequence \eqref{seq:leaf} exact.
\end{proof}

\begin{definition} \label{def:holonomies}
We call \emph{$n$-th holonomy} of the singular foliation $\fol{F} $ near the leaf $L$ the group morphisms \eqref{eq:def_Holonomy}:
  \begin{equation} \label{eq:FolHolonomy}
   Hol \colon  \left\{ \begin{array}{rcll} \pi_n(L,\ell) &\to & \Gamma( \pi_{n-1}(\fol{T}_\ell)) & \hbox{ for $n \geq 3$} \\
    \pi_2(L,\ell) &   \to  &  \mathrm{Center}(\Gamma(I(\mathbf T_\ell))  & \\
    \pi_1(L,\ell) & \to & {\rm Diff}( p^{-1}(\ell) / \fol{T}_\ell )  & \\ \end{array}\right. 
\end{equation}
We call \emph{higher holonomies} the family of these group morphisms and  \emph{higher holonomy long exact sequence} the long exact sequence \eqref{seq:leaf}.
\end{definition}

Above, we used the identification $ \Gamma(\pi_1(\fol{T}_\ell)) \simeq \Gamma(I(\mathbf T_\ell)) $ with $ \mathbf T_\ell$ the universal holonomy groupoid of $\fol{T}_\ell $ seen in Proposition 2.18,  and the identification $ \Gamma(\pi_0(\fol{T}_\ell)) \simeq  {\rm Diff}( p^{-1}(\ell) / \fol{T}_\ell ) $
of Definition \ref{def:pi0}. We also implicitly used Theorem  \ref{thm:snake} to assure that the image of $\pi_2(L,\ell)$ under $Hol$ is valued in the center of $\Gamma(I(\mathbf T_\ell)$.

\begin{remark}
\normalfont
Although the context is slightly different, the first holonomy $$Hol:\pi_1(L)\to \mathrm{Diff}(p^{-1}(\ell)/\fol{T}_\ell)$$ matches the holonomy map constructed by Dazord in \cite{Dazord}. Instead of complete Ehresmann connection,  \cite{Dazord} assumes saturated neighborhoods (called stable in his paper), but a line-by-line comparison shows that both morphisms agree for leaves that admit a neighborhood satisfying both conditions.
\end{remark}

Let $\tilde L$ be a leaf of $ \fol{F}$ through some point $m \in p^{-1}(\ell)$ for $ \fol{F} |_{M_L}$. 
The restricted map $ p|_{\tilde L} \colon \tilde{L} \to L $ is a surjective submersion, and any complete Ehresmann connection on $p \colon M_L \to L $ restricts to an Ehresmann connection on $p|_{\tilde L} $. By a classical theorem for fibrations (cf. eg. \cite{MR1867354}), there exists a long exact sequence:
\begin{equation}\label{eq:les-topology}  \dots \to \pi_n( L , \ell) \to  \pi_{n-1} ( \tilde{S}, m ) \to \pi_{n-1} ( \tilde{L}, m )  \to \cdots \end{equation}
with $\tilde S = p^{-1}(\ell) \cap \tilde{L}$ the fiber of $p|_{\tilde L} \colon \tilde L \to L $.
 \begin{proposition}
 \label{prop:regular_leaves}
 For $m$ a regular point, the higher holonomy long exact sequence \eqref{seq:leaf} coincides with the long exact sequence \eqref{eq:les-topology} associated to the fibration $p |_{\tilde L} \colon \tilde L \to L$.
 \end{proposition}
 \begin{proof}
Since both the leaf $\tilde L$ and the connected component of $m$ in  $ \tilde S$ are regular leaves of $ \fol{F}$ and $\fol{T}_\ell$, 
Example \ref{ex:regularleaf} implies
$ \pi_n(\fol{F},m)=\pi_n ( \tilde L , m )  \hbox{ and } \pi_n(\fol{F},m)=\pi_n (\tilde S,m ) $.
 \end{proof}

\subsubsection{The higher holonomies and Androulidalis-Skandalis holonomy groupoid.}

\label{sec:LinkWithHolonomyGroupoid}

We show that the $n$-th holonomy of the singular foliation $\fol{F} $ near the leaf $L$ ``projects'' to a similar group morphism associated to Androulidakis-Skandalis' holonomy Lie groupoid. 

Let $ {\fol{F}}$ be a singular foliation that admits a geometric resolution, and $L$ be a locally closed leaf that admits a complete Ehresmann ${\fol{F}}$-connection $(M_L,p,H)$.
Throughout this section, we denote by $\mathbf F$ the universal holonomy groupoid of $\fol{F} $ and by  $\mathbf T_\ell $ the universal holonomy
groupoid of the transverse singular foliation $ \fol{T}_\ell$ on the fiber $ p^{-1}(\ell)$.

We choose $m \in p^{-1}(L)$ and we denote by $ \tilde L$ and $\tilde S $ leaves through $m$ of $\fol{F} $ and $\fol{T}_\ell $ respectively.
Let us apply Theorem \ref{thm:snake} (or Theorem 1.4 in \cite{Zhu}) 
to the surjective submersion of Lie algebroids,
 $$ P :=  p \circ \rho : A_{\tilde L}  \mapsto   T L   . $$
 The fiber over $ \ell$ is easily identified with the holonomy Lie algebroid $ A_{\tilde S}$ of the leaf $ \tilde{\mathcal S}$. Also, the Ehresmann $ \fol{F}$-connection on $p: M_L \to L $ induces an Ehresmann connection on $P$ as in Proposition \ref{prop:compatEhresmann}. 
 By Proposition 2.2. in \cite{MR3119886}, $ \mathbf  T_\ell|_m$ and $\mathbf F|_m $ are smooth manifolds for all $m \in M$.
  {By Example \ref{ex:LieAlgebroid},} we obtain for all $m \in M$ an exact sequence of the form:
 \begin{equation} \label{exact:Gr} 
 \xymatrix{  \ar[r]\pi_n(\mathbf T_\ell|_m ,m ) &\ar[r]
      \pi_n(\mathbf F|_m  , m) &  \pi_n(L,\ell) \ar[dll]_{Hol_{AS}|_m}  \\
      \pi_{n-1}(\mathbf T_\ell|_m,m )&\cdots & \pi_2(L,\ell) \ar[dll]_{Hol_{AS}|_m}  \\ { I_m} (\mathbf T_\ell ) \ar[r] &  I_m(\mathbf F )\ar[r]   & \pi_1(L,\ell) \ar[dll]_{Hol_{AS}|_m}  \\ p^{-1}(\ell)/\fol{T}_\ell & &   }  \end{equation}
The connecting maps $ Hol_{AS}|_m$, altogether, form group morphisms:
  \begin{equation} \label{eq:GrHolonomy}
   Hol_{AS} \colon  \left\{ \begin{array}{rcll} \pi_n(L,\ell) &\to & \Gamma( \pi_{n-1}(\mathbf T_\ell)) & \hbox{ for $n \geq 3$} \\
    \pi_2(L,\ell) &   \to  &  \mathrm{Center}(\Gamma(I(\mathbf T_\ell)))  & \\
    \pi_1(L,\ell) & \to & {\rm Diff}( p^{-1}(\ell) / \fol{T}_\ell )  & \\ \end{array}\right. 
\end{equation}
Above, $ \Gamma(I(\mathbf T_\ell)) $ stands for sections of the source $ s: \mathbf T_\ell \to p^{-1}(L)$ which are for all $ m \in M$ valued in the isotropy group at $m$ of the groupoid $\mathbf T_\ell  $, a group that we denote by $ I_m(\mathbf T_\ell)$. Also, $ I_m(\mathbf F)$ stands for the isotropy group of $ \mathbf F$ at $m$.

\begin{definition}
We call $n$-th \emph{AS holonomy of $L$} the group morphism \eqref{eq:GrHolonomy}  
and \emph{AS long exact sequence near $L$} the exact sequence \eqref{exact:Gr}.
\end{definition}

We relate these connecting maps to the higher holonomies of the leaves in the next proposition.

\begin{proposition}\label{prop:ASrelation}
 Let $ {\fol{F}}$ be a singular foliation that admits a geometric resolution. Let $L$ be a locally closed leaf that admits a complete Ehresmann ${\fol{F}}$-connection $(M_L,p,H)$.
For every $\ell \in L$ and $n \in {\mathbb N}^*$, the $n$-th holonomy and the $n$-th AS holonomy of the leaf $L$ are equal for $n=1,2$ and for $n \geq 3$ are related by:
\begin{align}
\label{eq:factorize0}
\xymatrix@R-0.8pc{
      && \Gamma(\pi_{n-1}(\fol{T}_\ell )) \ar@{-->}[dd]\\
 \pi_n ( L , \ell )       \ar[rrd]_{Hol_{AS}}\ar@{->}[rru]^{Hol}
      & \\ & & \Gamma(\pi_{n-1}( \mathbf T_\ell )) }
\end{align}
and for all $m \in p^{-1}(\ell)$, we have a natural morphism of long exact sequences:
\begin{align}\label{eq:algsnake0}
\xymatrix@R-0.8pc{
      & \pi_n( {\fol{T}}_\ell,m ) \ar[r]^{i}  \ar@{-->}[dd]
      & \pi_n({\fol{F}}|_{M_L}  , m)  \ar[rd]^{P} \ar@{-->}[dd]
      & 
      & \pi_{n-1}( {\fol{T}}_\ell|_m,m )\ar[r] \ar@{-->}[dd]
      &\dots
      \\ \ar[rd]_{Hol_{AS}|_m} 
      \dots \ar[ru]^{Hol|_{m}}
      &
      & 
      &\ar[rd]_{Hol_{AS}|_m} \pi_n(L,\ell)  \ar[ru]^{Hol|_{m}} 
      &
      &\\
      &\pi_n(\mathbf T_\ell|_m ,m ) \ar[r]^{i}
      &\pi_n(\mathbf F|_m  , m)  \ar[ru]_{P_A}
      &  
      &\pi_{n-1}(\mathbf T_\ell|_m,m )\ar[r] 
      &\dots}
\end{align}
where all vertical lines above are the natural morphisms $q$ of Proposition \ref{prop:computingISotropies}, the upper line is the higher holonomy long exact sequence \eqref{eq:def_Holonomy} of $ \fol{F}$ near $L$ and  the lower one is the AS long exact sequence \eqref{exact:Gr} near $L$. 
\end{proposition}
\begin{proof}
Let $\tilde L$ be the leaf through $m \in p^{-1}(\ell)$.
For all $ \sigma \colon S^n \to L$,
consider the horizontal lift $ \mathcal L^H (\sigma, \underline{m}) \colon T[1](I \times S^n) \to \NQ{U}^\fol{F}$ in Fundamental Lemma \ref{lem:lifts}. Its image through $  \NQ{U}^\fol{F} \to A_{\tilde L}$ is the horizontal lift for the Lie algebroid fibration $A_{\tilde L} \to TL $. 
 This implies the commutativity of the diagrams \eqref{eq:factorize0} and \eqref{eq:algsnake0}.
\end{proof}

\begin{remark}
\normalfont
The restriction of $ \fol{F}$ to the open subset $M_{reg} \subset M_L $ of regular points is a regular foliation. Its holonomy groupoid (more precisely, its universal cover) is the restriction to $ M_{reg}$ of the universal holonomy groupoid $ \mathbf F$. This gives an alternative proof of Proposition \ref{prop:regular_leaves} as an immediate consequence and Proposition \ref{prop:ASrelation} and Example \ref{ex:regularleaf}.
\end{remark}

\begin{remark}
\normalfont
Iakovos Androulidakis and Marco Zambon \cite{AZ13,AZ2} have also defined a holonomy morphism for a leaf of a singular foliation, which goes from the holonomy groupoid to some quotients of jets of vector fields on the transverse singular foliation. It is not a higher holonomy in our sense. However, it can be related to our construction as follows: Any element of the holonomy groupoid comes from a path in $A_L$ from $\ell$ to $\ell'$. Using the Ehresmann $\fol{F}$-connection, for any $m\in p^{-1}(\ell)$ this path can be lifted to a path ending in $\phi(m)\in p^{-1}(\ell')$. The class of the map $\phi$ modulo $\mathrm{exp(I_\ell\fol{T}_\ell)}$ is the holonomy of \cite{AZ13,AZ2}. Alfonso Garmendia and Joel Villatoro indeed reconstructed the holonomy groupoid using such paths, see \cite{GarVilt}.
\end{remark}

\subsubsection{The higher holonomies for singular foliations coming from a Lie algebroid.}

\label{sec:IfLieAlgebroid}

We show that if there exists a Lie algebroid defining the singular foliation $ \fol{F}$, then the  $n$-th holonomy is obtained by projection from a group morphism associated to this algebroid, introduced by Olivier Brahic and Chenchang Zhu in \cite{Zhu}, that we call the  $n$-th BZ holonomy of $ \fol{F}$ near~$L$.

Throughout this section, we consider an integrable Lie algebroid $(A,[ \cdot, \cdot ]_A, \rho_A)$ over $M$. We denote by $\mathbf A $ its fundamental groupoid (also referred to as the Weinstein groupoid). Then $ {\fol{F}} = \rho_A(\Gamma(A))$  is a singular foliation.

Let us adapt the main results of \cite{Zhu} in our context: Assume that a locally closed leaf $ L$ comes equipped with a complete Ehresmann connection $ (M_L,p,H)$. 
By construction, $ P_A:= p \circ \rho_A \colon A \to TL$  is a surjective submersion and a Lie algebroid morphism. Its kernel over a point $\ell \in L$ is the sub-Lie algebroid  $(  K , [\cdot, \cdot]_A ,\rho_ A )$  with  $   K|_m  = {\mathrm{Ker}}(T_m p \circ \rho_A|_m) $. 
We denote by $\mathbf K $ its fundamental Lie groupoid.
  Theorem \ref{thm:snake} or Theorem 1.4 in \cite{Zhu}, together with Example \ref{ex:LieAlgebroid}, can then be applied to yield a family of group morphisms:
 \begin{align}
\label{eq:AHolonomy} 
\delta_{BZ} \colon  \left\{
\begin{array}{rcll} \pi_n(L,\ell) &\to & \Gamma( \pi_{n-1}(\mathbf K_\ell)) & \hbox{ for $n \geq 3$} \\
    \pi_2(L,\ell) &   \to  &  \mathrm{Center}(\Gamma(I(\mathbf K_\ell)))  & \\
    \pi_1(L,\ell) & \to & \mathrm{Diff}( p^{-1}(\ell) / \fol{T}_\ell )  & \\ \end{array}\right. 
\end{align}
making for all $ m \in p^{-1}(\ell)$ the following long sequence exact:
\begin{align}
\label{exact:BZ}
\xymatrix{  \ar[r]\pi_n(\mathbf K_\ell|_m ,m ) &\ar[r]
      \pi_n(\mathbf A|_m  , m) &  \pi_n(L, \ell) \ar[dll]_{\delta_{BZ}|_m}  \\
      \pi_{n-1}(\mathbf K_\ell|_m,m )&\cdots & \pi_2(L,\ell) \ar[dll]_{\delta_{BZ}|_m}  \\ { I_m} (\mathbf K_\ell ) \ar[r] &  I_m(\mathbf A_\ell )\ar[r]   & \pi_1(L,\ell) \ar[dll]_{\delta_{BZ}|_m}  \\ p^{-1}(\ell)/\fol{T}_\ell & &   }
\end{align}

\begin{definition}
Let $(A,[ \cdot, \cdot ]_A, \rho_A) $ be a Lie algebroid with a complete Ehresmann connection over $M$. 
We call $n$-th \emph{BZ holonomy\footnote{BZ stands for Olivier Brahic and Chenchang Zhu.} of $A$ near $L$} the group morphism \eqref{eq:AHolonomy}  
and \emph{BZ long exact sequence of $A$ near $L$} the long exact sequence \eqref{exact:BZ}.
\end{definition}

Corollary 3.12 in \cite{Zhu} describes the evaluation at $ \ell$ of the $BZ$-holonomy. Here is the relation between both family of holonomies:
\begin{proposition}
\label{prop:BZ}
Let $A$ be an integrable Lie algebroid with fundamental groupoid $\mathbf A $ and $ {\fol{F}} = \rho_A(\Gamma(A))$ be its induced singular foliation.
For every locally closed leaf $L$ that admits a complete Ehresmann ${\fol{F}}$-connection $(M_L,p,H)$, the first holonomy and the first BZ-holonomy coincide, and for $n \geq 2$ the $n$-th holonomy of $\fol{F}$ near $L$ factors through the $n$-th BZ-holonomy of $A$ near $L$:
\begin{align}
\label{eq:factorize}
\xymatrix@R-0.8pc{
      &&\Gamma(\pi_{n-1}(  \mathbf K_\ell ))\ar@{-->}[dd]\\
 \pi_n ( L , \ell )       \ar[rru]^{\delta_{BZ}}\ar@{->}[rrd]_{Hol}
      & \\ & &\Gamma(\pi_{n-1}(\fol{T}_\ell )) }
      \xymatrix@R-0.8pc{
      &&{\mathrm{Center}} \left(\Gamma(I(\mathbf K_\ell ))\right) \ar@{-->}[dd]\\
 \pi_2 ( L , \ell )       \ar[rru]^{\delta_{BZ}}\ar@{->}[rrd]_{Hol}
      & \\ & &{\mathrm{Center}} \left( \Gamma(\pi_1(\fol{T}_\ell )) \right)}
\end{align}
For all $m \in p^{-1}(\ell)$, moreover, we have a natural morphism of long exact sequences:
\begin{align}\label{eq:algsnake}
\xymatrix@R-0.8pc{
      &\pi_n(\mathbf K_\ell|_m ,m ) \ar[r]^{i}\ar@{-->}[dd] 
      &\pi_n(\mathbf A|_m  , m)  \ar@{-->}[dd] \ar[rd]^{P_A}
      &  
      &\pi_{n-1}(\mathbf K_\ell|_m,m )\ar[r]\ar@{-->}[dd]  
      &\dots\\ \ar[ru]^{\delta_{BZ}|_m} 
      \dots \ar[rd]_{Hol|_{m}}
      &
      & 
      &\ar[ru]^{\delta_{BZ}|_m} \pi_n(L,\ell)  \ar[rd]_{Hol|_{m}} 
      &
      &\\
      & \pi_n( {\fol{T}}_\ell,m ) \ar[r]^{i} 
      & \pi_n({\fol{F}}|_{M_L}  , m)  \ar[ru]_{P}
      & 
      & \pi_{n-1}( {\fol{T}}_\ell,m )\ar[r]
      &\dots}
\end{align}
where the upper line is the BZ long exact sequence \eqref{exact:BZ} of $A$ near $L$ and the lower one is the higher holonomy long exact sequence \eqref{eq:def_Holonomy} of $ \fol{F}$ near $L$. 
\end{proposition}
\begin{proof}
By the universality Theorem \ref{theo:onlyOne}, there exists a (unique up to homotopy) NQ-morphism 
 $\Phi\colon A[1] \to  \NQ{U}^\fol{F} $ and a unique NQ-morphism 
  $  {\NQ{T}_\ell}\to \NQ{U}^{{\fol{T}}_\ell} $ (which can chosen to be the restriction of $ \Phi$ to $ p^{-1}(\ell)$). If $L(m,\sigma)$ is the lift in the Fundamental Lemma \ref{lem:lifts} applied to a base map $ \sigma : S^n \to L$ and an initial shape $e = \underline{m}$ for the fibration $P_A \colon A \to TL $, then $ \Phi \circ L(e,m) $ is the lift for the same   base map $ \sigma $ and the same initial shape $ \underline{m}$ for the NQ-manifold fibration $ P : \NQ{U}^{\fol{F}} \to T[1]L$.  This implies that $ \Phi$ and $ \Phi|_{p^{-1}(\ell)}$ are NQ-morphisms whose induced vertical arrows of \eqref{eq:factorize} and \eqref{eq:algsnake} satisfy all required conditions. 
\end{proof}

 One of the most intriguing open question about singular foliations, that first appeared in \cite{AZ13}, is \emph{``Let $ {\fol{F}}$ be a finitely generated singular foliation, do we have $ {\fol{F}} = \rho_A( \Gamma(A))$ for some Lie algebroid $ (A,[ \cdot, \cdot ]_A, \rho_A)$? Is it at least true locally?''}. Let us recapitulate a few points about this problem (see Example 3.9--3.16 in \cite{LLS}):
 \begin{enumerate}
   \item  For many singular foliations, no natural Lie algebroid  seems to exist (\emph{e.g. }vector fields on $ \mathbb R^n$ vanishing at order two at zero, or tangent to a given affine variety).
     \item For many singular foliations, such a Lie algebroid is known to exist  (\emph{e.g. }symplectic leaves of a Poisson structure, orbits of a Lie group action).
     \item  But even if such a Lie algebroid exists, it may not be unique. 
 \end{enumerate}
 For singular foliations as in item 2, Proposition \ref{prop:BZ} is useful to compute the holonomies.  But item 1 warns us that it might not be always possible. Item 3 means that we should not hope that the BZ holonomies ``know more'' about the singular leaf than the holonomies that we constructed: they certainly ``know more'' in view of (\ref{eq:factorize}), but they may contain information which is not related to the singular foliation itself but only to the particular Lie algebroid defining it.

\subsection{Examples and particular cases}
\label{sec:exsing}

Let us specialize our construction of the higher holonomies to several types of singular foliations.

\vspace{0.2cm}
{\textbf{Projective singular foliations.}} We say that a singular foliation $\fol{F} $ is \emph{Debord} or \emph{projective} when $\fol{F} $ is projective as a module over $ \Cinfty{M}$. By the Serre-Swann theorem, it means that $ {\fol{F}}$ is the image of a Lie algebroid $A$ through an anchor map which is injective on a dense open subset. Since the $\Cinfty{M}$-module of sections of the Lie algebroid $T{\fol{F}}$ is isomorphic to the $\Cinfty{M}$-module $ {\fol{F}}$, the Lie algebroid $A[1]$ is a universal NQ-manifold of $ {\fol{F}}$.
The following is an obvious consequence of Proposition \ref{prop:BZ}:

\begin{corollary}
	Let $ {\fol{F}}$ be a Debord singular foliation with associated Lie algebroid $A$, and $L$ a leaf that admits an Ehresmann connection on a neighborhood. The higher holonomies and the higher holonomy long exact sequence of ${\fol{F}} $ at  $L$ coincide with the BZ holonomies and the BZ long exact sequence of the Lie algebroid $A$ at $L$.
\end{corollary}

Of course, regular foliations are an instance of this situation.

\begin{example}
\normalfont
    For a regular foliation, only $ Hol_1$ is not zero, and coincides with the usual holonomy.
\end{example}

\vspace{0.2cm}
{\textbf{Singular foliations arising from a Lie algebroid of minimal rank.}}  Let $L$ be a leaf of a singular foliation that admits an Ehresmann connection $ (M_L,p,H)$, and $A_L$ its holonomy Lie algebroid.
We say that the leaf $L $ admits a Lie algebroid of minimal rank if there exists a Lie algebroid $A$ on $M_L$ defining $\fol{F}$ such that $A|_L=A_L$.
 
 Proposition \ref{prop:ASrelation} states that $Hol_{AS}$ and $Hol$ agree on $ \pi_1(L,m)$ and $\pi_2(L,m) $ for any leaf $L$ of a singular foliation that admits a complete Ehresmann connection. 
 For the higher holonomies, the following is an immediate consequence of Proposition \ref{prop:BZ}:

\begin{corollary}
\label{cor:minimalRank}
    If $L$ admits a Lie algebroid of minimal rank $A$, then there exist for all $ n \geq 3$ group morphisms $\phi: \pi_{n-1}( \mathbf T_\ell|_m ,m  )\to \pi_{n-1} ( \fol{T}_\ell , m)$ such that $Hol= \phi \circ Hol_{AS}$.
\end{corollary}

Here are a few instances of this situation:

\begin{example} 
\label{ex:vectorbundle}
\normalfont
Let $p:E\to L$ be a vector bundle of rank $k$ over a manifold $L$ that we identify with the zero section. For simplicity, we assume $\pi_1(L)=0$. Let $\fol{F}$ the singular foliation of all vector fields on $E$ tangent to $L$. Every linear connection induces an Ehresmann connection for the zero section.

The leaf $L$  admits a Lie algebroid of minimal rank: the Lie algebroid $CDO(E)$ (see \cite{Mackenzie}), so that we are in the situation of Corollary \ref{cor:minimalRank}.

\begin{enumerate}
    \item Let us compute $Hol_{AS}|_m $  for all nonzero $ m \in p^{-1}(L)$.  
    The restriction of the universal holonomy groupoid to $ E \backslash L $ is the fundamental groupoid of  $E \backslash L$. Since $L$ is simply-connected by assumption, the latter groupoid is the pair groupoid $(E \backslash L) \times ( E \backslash L ) \rightrightarrows (E \backslash L)$, and  $Hol_{AS}|_m $ is the connecting map associated to the fibration:
 $  p \colon E \backslash L \to L $. 
Since the fiber is $ E_\ell \backslash \{0_\ell \}$. The latter being homotopic to a $(k-1)$-sphere, with $k$ the rank of $E$, we obtain group morphisms:
  $$ Hol_{AS}|_m : \pi_n(L,\ell) \to \pi_{n-1}(S^{k-1},N) $$
    \item Let us compute $Hol_{AS}|_\ell $ for $\ell\in L$.   Since the Lie algebroid $CDO(E)$ integrates to the Lie groupoid made of all invertible linear maps between fibers of $E$, the universal cover of $s^{-1}(\ell) $ is the universal cover $\widetilde{{\mathrm{Fr}}(E)}$ of the frame bundle of $E$. Its fiber over $ \ell$ is the universal cover $ \widetilde{GL(E_\ell)}$ of $ GL(E_\ell)$.  In this case, $Hol_{AS}|_\ell $ is the connecting map associated to the fibration:
 $  p \colon \widetilde{{\mathrm{Fr}}(E)}  \to L $.
 Since the fibers are $ \widetilde{GL(E_\ell)}$, we obtain therefore group morphisms:
  $$ Hol_{AS}|_\ell := \left\{\begin{array}{rcll} \pi_n(L,\ell)& \to& \pi_{n-1}( \widetilde{GL(E_\ell)},e)=\pi_{n-1}(GL(E_\ell),e ) & \hbox{ for $n \geq 3$ }  \\ 
 \pi_2(L,\ell) &  \to & {\mathrm{Center}} \left( \widetilde{GL(E_\ell)}\right) = \mathbb Z/ 2 \mathbb Z &   \end{array} \right. $$
 A direct computation shows that the second map corresponds to the second class of Stiefel-Whitney. 
\end{enumerate}
\end{example}

\begin{example}
\normalfont
When $L=S^k$ and $ E= TS^k$ with $ k \geq 2$, the higher holonomies of Example \ref{ex:vectorbundle} satisfy $Hol_{AS}|_m=0$ if and only if
\begin{enumerate}
    \item[(i)]  $m$ is not in the zero section of $TS^n$ and $k=3,7$.
    \item[(ii)] $m$ is in the zero section of $TS^n$ and $k$ is odd.
\end{enumerate}

This comes from the well-known fact that the tangent spaces of spheres of odd dimension admit a nowhere vanishing section, while their frame bundles admit a section iff $k=3,7$.  
\end{example}

\begin{example}
``Period map of Crainic-Fernandes.''
\normalfont
The second higher holonomy, evaluated at the point $\ell $, is a group morphism $\pi_2(L) \to {\mathrm{Center}} (I_\ell(\mathbf T_\ell))$. Let us check that it coincides with Crainic-Fernandes' period map \cite{MR1973056} for the holonomy Lie algebroid $ A_L \to L$. The higher holonomies $ Hol|_\ell $ and  $Hol_{AS}|_\ell $ agree on $\pi_2(L,\ell) $ in view of Proposition \ref{prop:ASrelation}. 
   Item 2 in Corollary 3.12 in \cite{Zhu} identifies $ Hol_{AS}|_\ell$ (which is associated to the Lie algebroid fibration $ A_L \to TL$) with the period map of Crainic-Fernandes.
This completes the proof of the claim.
We could also prove the claim using the equality of the short exact sequence \eqref{eq:computingISotropies3} 
with the short exact sequence (3.6) in \cite{AZ13} and a similar claim made in Remark 4.3 in \cite{AZ13}. Lastly, since Debord \cite{MR3119886} proved longitudinal smoothness (using a totally independent argument),  the period map of Crainic-Fernandes has a discrete image.
\end{example}

\bibliographystyle{alpha}
\bibliography{biblio}	

\end{document}